\documentclass{article}
\usepackage{hyperref}
\usepackage[utf8]{inputenc}
\usepackage[cyr]{aeguill}
\usepackage[all]{xy}
\usepackage[french, english]{babel}
\usepackage[T1]{fontenc}
\usepackage{BOONDOX-cal}
\usepackage{amscd}
\usepackage{amsmath}
\usepackage{amsthm} 
\usepackage{amssymb}
\usepackage{cleveref}
\usepackage{titlesec}
\usepackage{indentfirst}
\usepackage{enumitem}
\newtheorem{theo}{Theorem}[subsubsection]
\newtheorem{theointro}{Theorem}

\newtheorem*{theo*}{Theorem}
\theoremstyle{definition}  
\newtheorem{defi}[theo]{Definition}
\newtheorem*{defi*}{Definition} 
\newtheorem{defis}[theo]{Definitions}                  
\newtheorem{lemm}[theo]{Lemma}
\newtheorem{prop}[theo]{Proposition}
\newtheorem{coro}[theo]{Corollary}
\newtheorem{propintro}[theointro]{Proposition}
\newtheorem*{prop*}{Proposition}
\theoremstyle{remark}

\newtheorem*{ex*}{Example}
\newtheorem{rema}[theo]{Remark}
\newtheorem*{rema*}{Remark}
\newtheorem*{Note*}{Notation}
\newtheorem{note}[theo]{Notation}
\newtheorem*{note*}{Notation}

\renewcommand\thesubsubsection{\ifnum\arabic{subsection}=0\relax{\thesection}\else{\ifnum\arabic{subsubsection}=0\relax{\thesubsection}\else{\thesubsection.\arabic{subsubsection}}\fi}\fi}
\newcommand{\overeg}[1]{\overset{\mbox{#1}}{=}}
\titleformat{\section}
     {\center\Large\scshape}{\thesection.}{1em}{}
 \titleformat{\subsection}
     {\center\large\scshape}{\thesubsection.}{1em}{}[]
 \titleformat{\subsubsection}
     {\scshape}{\thesubsubsection.}{1em}{}
\def\scr{\mathcal}                    
\def\p{\partial}
\def\Op{\operatorname{Op}}
\def\Ind{\operatorname{Ind}}
\def\BHS{\operatorname{BHS}}
\def\N{\mathbb N}
\def\diag{\shorthandoff{;:!?}
\xymatrix}
\def\0{\mathbf 0}
\def\_{\underline}
\def\t{\mathfrak t}
\def\u{\underline{u}}
\def\v{\underline{v}}
\def\P{\mathcal P}
\def\L{\scr L}
\def\Z{\mathbb Z}
\def\O{\scr O}
\def\F{\mathbb F}
\def\B{\mathcal B}
\def\M{\mathcal M}
\def\ı{\mathcal N}
\def\r{{\_r}}
\def\q{{\_q}}
\def\C{\scr C}
\def\∑{\mathfrak S}
\def\to{\rightarrow}
\def\inj{\hookrightarrow}
\def\surj{\twoheadrightarrow}
\def\val#1{\left\vert#1\right\vert}

\def\Comp{\operatorname{Comp}}
\def\undercount#1#2{\underset{#1}{\underbrace{#2}}}

\author{Sacha Ikonicoff}
\date{\today}
\title{\vspace{-52pt}Divided power algebras over an operad}
\begin{document}
\maketitle
\thispagestyle{empty}
\begin{abstract}
The purpose of this paper is to give a characterisation of divided power algebras over a reduced operad. Such a characterisation is given in terms of polynomial operations, following the classical example of divided power algebras.

We describe these polynomial operations in two different ways: one way uses invariant elements under the action of the symmetric group, the other coinvariant elements. Our results are then applied to the case of level algebras, which are (non-associative) commutative algebras satisfying the exchange law.
\end{abstract}
\tableofcontents
\section{Introduction and acknowledgements}
\subsection{Introduction}
The study of what has been called ``divided power'' algebra structures dates back to the 1950s. This term can be found in Norbert Roby's work (e.g. \cite{NR}) and in Henry Cartan's presentation for his seminar (in \cite{HC}). What was then called divided power algebra would now be called, with the language of operads, a $\Gamma(Com)$-algebra, or a $Com$-algebra with divided powers (a new proof of this fact is given in Section \ref{secCOM}). In \cite{HC}, Henri Cartan shows that the homotopy of a simplicial commutative algebra is endowed with a structure of (graded) $\Gamma(Com)$-algebra. More recently, Benoît Fresse proved in \cite{BF} a more general result: the homotopy of a simplicial algebra over an operad $\P$ admits a structure of (graded) $\Gamma(\P)$-algebra. Unfortunately, only very little is known about the structure of $\Gamma(\P)$-algebras, given an arbitrary reduced operad $\P$. In $\cite{BF}$, Benoît Fresse describes the structure of a $\Gamma(\P)$-algebra when $\P$ is one of the classical operads $As, Com, Lie$ and, in characteristic 2, also when $\P$ is the operad $Pois$. Ioannis Dokas made explicit some non-trivial structures associated to $\Gamma(PreLie)$-algebras in \cite{ID}. More recently, a complete explicit definition of the structure of a $\Gamma(PreLie)$-algebras in terms of brace operations was obtained by Andrea Cesaro (in \cite{AC}).

The main goal of this article is to give a characterisation of divided power algebras over a reduced operad $\P$ in terms of polynomial operations bound by a list of relations, following the example of the classic definition of divided power algebras (in \cite{NR}). We shall first prove the following result:
\begin{theointro}[\Cref{theoinv}]\label{theoinvintro}
Let $\P$ be a reduced operad. The category of $\Gamma(\P)$-algebras is isomorphic to the category $\beta'(\P)$ of vector spaces $A$ endowed with a family of polynomial operations $\beta_{x,\r}:A^{\times p}\to A$ where $\r=(r_1,\dots,r_p)$ is a finite list of integers and $x\in \P(r_1+\dots+r_p)$ is stable under the action of $\∑_{r_1}\times\dots\times \∑_{r_p}\subset\∑_{r_1+\dots+r_p}$, and which satisfy some 8 relations (see \Cref{defbeta'}).
\end{theointro}
We shall also prove the following result:
\begin{theointro}[\Cref{theocoinv}]\label{theocoinvintro}
  Let $\P$ be a reduced operad spanned by a set operad. The category of $\Gamma(\P)$-algebras is isomorphic to the category $\gamma'(\P)$ of vector spaces $A$ endowed with a family of polynomial operations $\gamma_{[x],\r}:A^{\times p}\to A$ where $\r=(r_1,\dots,r_p)$ is a list of integers and $[x]\in \P(r_1+\dots+r_p)_{\∑_{r_1}\times\dots\times\∑_{r_p}}$, and which satisfy some 7 relations (see \Cref{defgamma'}).
\end{theointro}

The operations $\beta_{*,*}$ and $\gamma_{*,*}$ from these theorems are to be seen as ``divided power operations''. The practical interest of the second result is that, in general, it is easier to describe the set of orbits of a $G$-set than the invariant submodule of a representation of $G$.

The proof of these results also uncovers two results involving $\∑$-modules without operad structure:
\begin{propintro}[\Cref{propinv}]\label{propinvintro}
  Let $\M$ be a $\∑$-module. The category $\Gamma(\M)_{mod}$ of vector spaces $A$ equipped with a morphism $\bigoplus_{n\ge0}(\M(n)\otimes A^{\otimes n})^{\∑_n}\to A$ is isomorphic to the category $\beta(\M)$ of vector spaces $A$ endowed with a family of polynomial operations $\beta_{x,\r}:A^{\times p}\to A$ where $\r=(r_1,\dots,r_p)$ is a finite list of integers and $x\in \M(r_1+\dots+r_p)$ is stable under the action of $\∑_{r_1}\times\dots\times \∑_{r_p}\in\∑_{r_1+\dots+r_p}$, and which satisfy some 6 relations (see \Cref{defbeta}).
\end{propintro}
\begin{propintro}[\Cref{propcoinv}]\label{propcoinvintro}
  Let $\M$ be a $\∑$-module equipped with a basis stable under the action of the symmetric groups (that is, $\M$ is spanned by a $\∑$-set). The category $\Gamma(\M)_{mod}$ is also isomorphic to the category $\gamma(\M)$ of vector spaces $A$ endowed with a family of polynomial operations $\gamma_{[x],\r}:A^{\times p}\to A$ where $\r=(r_1,\dots,r_p)$ is a finite list of integers and $[x]\in \M(r_1+\dots+r_p)_{\∑_{r_1}\times\dots\times\∑_{r_p}}$, and which satisfy some 5 relations (see \Cref{defgamma}).
\end{propintro}
More precisely, under the assumptions of the previous theorems and propositions, one has the following diagram of categories:
$$\diag{
  \Gamma(\M)_{mod}\ar@/_/[d]_{F_\M}&\Gamma(\P)_{alg}\ar@{_{(}->}[l]\ar@/_/[d]_{F_\P}\\
  \beta(\M)\ar@/_/[d]_{P_\M}\ar@/_/[u]_{G_\M}&\beta'(\P)\ar@{_{(}->}[l]\ar@/_/[d]_{P_\P}\ar@/_/[u]_{G_\P}\\
  \gamma(\M)\ar@/_/[u]_{O_\M}&\gamma'(\P)\ar@{_{(}->}[l]\ar@/_/[u]_{O_\P}
},$$
where the horizontal arrows are inclusions of full subcategories and exist if $\P=\M$ is endowed with a structure of reduced operad. The functors $G_*$ and $O_*$ are inverse to the corresponding functors $F_*$ and $P_*$.

\vskip 1cm
The notion of level algebras appeared in the study of unstable modules over the Steenrod algebra. In \cite{DD}, D.M. Davis defines non-associative ``depth-invariant'' algebras. Unstable modules over the Steenrod algebra are studied into more details in Lionel Schwartz's work \cite{LS}. Later, David Chataur and Muriel Livernet gave the definition of a level algebra (in \cite{CL}) and described the operad $Lev$, whose algebras are the level algebras. This definition coincides with that of Davis's depth-invariant algebras. In this article, we shall give a complete description of the structure of a $\Gamma(Lev)$-algebra. The operad $Lev$ being spanned by a set operad, \Cref{theocoinvintro} gives a characterisation of $\Gamma(Lev)$-algebras in terms of polynomial operations, and we will give a finer characterisation. Our result can be reformulated as follows:
\begin{theointro}[\Cref{theoL}]\label{theoLintro}
  A $\Gamma(Lev)$-algebra is an $\F$-vector space $A$ endowed with a family of polynomial operations $\phi_{h,\r}:A^{\times p}\to A$, where $\r=(r_1,\dots,r_p)$ is a list of integers, and $h:\{1,\dots,r_1+\dots+r_p\}\to\N$ is a map that is constant on each intervals $\{r_1+\dots, r_{k-1}+1,\dots, r_1+\dots+r_{k}-1\}$, $k=1,\dots,p$, and satisfies $\sum_{i\in[n]}2^{-h(i)}=1$, and which satisfy 7 relations (see \Cref{defL})
\end{theointro}

In Section \ref{Background}, we give the notation used in this article, as well as definitions regarding operad theory.

In Section \ref{secinv} we give a construction of the categories $\beta(\M)$ and $\beta'(\P)$, and of the functors $F_\M,G_\M,F_\P$ and $G_\P$ of the previous diagram. We prove that those two couples of functors form isomorphisms of categories and we apply this result to the classical example of the operad $Com$ of associative commutative algebras.

In Section \ref{seccoinv} we give the construction of the categories $\gamma(\M)$ and $\gamma'(\P)$ and of the functors $P_\M,P_\P,O_\M$ and $O_\P$ of the previous diagram. 

In Section \ref{secLEV}, we prove \Cref{theoLintro}, as a refinement of \Cref{theocoinvintro}, in the case $\P=Lev$. 

In Section \ref{secStep}, we study the connections between $\Gamma(Lev)$-algebras and other types of algebras into more details.

\Cref{A} is devoted to the explicit computation of the multiplication $\tilde \mu$ of the monad $\Gamma(\P)$ for a reduced operad $\P$.

\begin{Note*}\item
\begin{itemize}
  \item The ground field $\F$ is fixed for the whole article.

  \item In order to make this article more comprehensible, $\∑$-modules and operads, as well as their bases, will be referred to with script letters (except when using Greek letters), whereas vector spaces and their bases will be assigned print letters. 
\end{itemize}
\end{Note*}
\subsection{Acknowledgements}

I would like to thank Muriel Livernet for advising me during the whole redaction process of this article. I would like to thank Benoît Fresse, the conversations that we had on the subject of divided powers have helped me sort some of the difficult technical issues that arose in the proofs of my results. I would also like to thank Eric Hoffbeck, who took the time to give my article a first reading. Finally, I would like to warmly thank the anonymous referee who gave me strong pieces of advice that helped drastically improve the layout of this article, and has given me a key mathematical remark to generalise my main result.

\section{Background}\label{Background}

In this section, we study the partitions of the set $\{1,\dots,n\}$, we give definitions for operations on those partitions and then recall the basic setting of operad theory on sets and on vector spaces.

First, we explain the definition of operations on the partitions of the sets $[n] = \{1,\dots,n\}$ which reflect the shape of the composition operations associated to the structure of an algebra over an operad.
The first subsection of the section, \Cref{Partitions}, is devoted to this topic. In passing, we also explain our notation and conventions for the symmetric groups and for partitions. In what follows, we use the definitions of this subsection in order to formulate the generalised divided power operations associated to operads. In particular, the notation which we introduce in \Cref{Partitions} are used freely throughout the paper.

In a second step, in \Cref{B}, we study the relationship between the coinvariants and the invariants for the vector space spanned by a $G$-set, where $G$ is any group. To be specific, we check that the coinvariants and the invariants coincide in this case. This observation is the crux of our description of the structure of a $\Gamma(\P)$-algebra for an operad spanned by an operad in sets.

Then, in \Cref{modop}, we review fundamental definitions of the theory of operads for operads in vector spaces and the general definition of the monads $S(\P)$ and $\Gamma(\P)$ associated to an operad $\P$.


\subsection{Partitions}\label{Partitions}

In this subsection, we explain our conventions on permutations, partitions, and we explain the definition of operations on partitions, which reflect the shape of composition operations associated to operads.

The operations are encoded by the structure of a coloured operad in sets. Therefore, we recall the definition of such an operad structure before tackling the applications to partitions.


\begin{Note*}\item
  \begin{itemize}
  \item Denote by $[n]$ the set $\{1,\dots,n\}$.

  \item For a set $E$, $\∑_E$ is the group of bijections $E\to E$, and $\∑_n$ is the group $\∑_{[n]}$.

  \item For $R=(R_i)_{1\le i\le p}$ a partition of $[n]$, define $\∑_R:=\prod_{i=1}^p\∑_{R_i}\subset\∑_n$. 

  \item For $X$ a $\∑_n$-set, for $x\in X$ and for $R$ a partition of $[n]$, denote by $[x]_R$, and $Stab_R(x)$, the class and the stabiliser of $x$ under the action of the subgroup $\∑_R$ of $\∑_n$. 

  \item For $G$ a group, $X$ a $G$-set and $x\in X$, denote by $\Omega_G(x)=\{g\cdot x\ |\ g\in G\}$ the orbit of $x$ under the action of $G$.

  \item For $E\subseteq[n]$ and $\rho\in\∑_n$, denote by $\rho(E)$ the set $\{\rho(x):x\in E\}$. 

  \item For $E\subseteq\N$ and $n\in\N$, denote $E+n:=\{x+n:x\in E\}$.
  \end{itemize}
\end{Note*}
\begin{defi}[\cite{BB} Definition 1.7]\label{defopcol}\item
   An $E$-coloured set operad $\P$ is the data, for all $(x_1,\dots,x_p,y)\in E^{p+1}$, of a set $\P(x_1,\dots,x_p;y)$ endowed, for all $\sigma\in\∑_p$, with a map $\P(x_1,\dots,x_p;y)\to\P(x_{\sigma^{-1}(1)},\dots,x_{\sigma^{-1}(p)},y)$, as well as an ``identity'' element $c_x\in\P(x;x)$ for all $x\in E$, and composition maps:
  $$\circ_i:\P(x_1,\dots,x_p;y)\times\P(z_1,\dots,z_q;x_i)\to \P(x_1,\dots,x_{i-1},z_1,\dots,z_q,x_{i+1},\dots,x_p;y),$$
  and that satisfy natural conditions of unitality, associativity and equivariance.
\end{defi}
\begin{defi}
  The $\N$-coloured set operad $\Pi$ is given by:
  $$\hspace{-12pt}\Pi(r_1,\dots,r_p;n)=\left\{\begin{array}{ll}
    \emptyset,&\mbox{ if }r_1+\dots+r_p\neq n,\\
    \{R=(R_i)_{1\le i\le p}:\sqcup_{i=1}^pR_i=[n]\mbox{ and }\forall i\in[p]\quad \val{R_i}=r_i\},&\mbox{ else.}
  \end{array}\right.
  $$
  The identity element $c_n\in\Pi(n;n)$ is the coarse partition $\big([n]\big)$. Given a permutation $\sigma\in\∑_p$, the map $\Pi(r_1,\dots,r_p,n)\to\Pi(r_{\sigma^{-1}(1)},\dots,r_{\sigma^{-1}(p)};n)$ sends $I$ to $\sigma\cdot I:=(I_{\sigma^{-1}(i)})_{i\in[p]}$. Compositions are given by refinement: if $R\in\Pi(r_1,\dots,r_p;n)$ and $Q\in\Pi(q_1,\dots,q_s;r_i)$, there exists a unique increasing bijection $b:[r_i]\to R_i$, and:
  $$R\circ_iQ=(R_1,\dots,R_{i-1},b(Q_1),\dots,b(Q_s),R_{i+1},\dots,R_p).$$
\end{defi}
\begin{rema}\label{transn}
  The set $\Pi(r_1,\dots,r_p;n)$ is also endowed with an action of $\∑_n$: if $R=(R_i)_{i\in [p]}\in \Pi(r_1,\dots,r_p;n)$, then $\rho(R)=(\rho(R_i))_{i\in[p]}$. This action is clearly transitive. 
\end{rema}
\begin{note}\item
  \begin{itemize}\label{COMP}
    \item A $p$-tuple of non-negative integers $\r=(r_1,\dots,r_p)$ such that $r_1+\dots+r_p=n$ is called a composition of the integer $n$. Denote by $\Comp_p(n)$ the set of compositions of $n$ into $p$ parts and $\Pi_p(n)$ the set of partitions of the set $[n]$ into $p$ parts. There is an injection $\iota:\Comp_p(n)\inj \Pi_p(n)$, mapping $\r$ to:
    $$r_1+\dots+r_{i-1}+[r_i]=\{r_1+\dots+r_{i-1}+1,\dots, r_1+\dots+r_i\}.$$
    The composition $\r$ will be identified with its image $\iota(\r)$. the map $\iota$ admits a left inverse $pr:\Pi_p(n)\surj\Comp_p(n)$, which sends $R=(R_i)_{i\in[p]}$ to $(\val{R_i})_{i\in[p]}$.
    \item $\Comp_p(n)$ is endowed with the following $\∑_p$-action: if $\sigma\in\∑_p$ and $\r\in\Comp_p(n)$, $\r^\sigma:=(r_{\sigma^{-1}(1)},\dots,r_{\sigma^{-1}(p)})$. Note that $\iota$ is not compatible with this action.
    \item In the set $\Pi(\r;n)$ there is a unique element of $\Comp_p(n)$ (more precisely, in the image of $\Comp_p(n)$ by $\iota$), namely $\r$, and $\Pi_p(n)=\bigsqcup_{\r\in \Comp_p(n)}\Pi(\r;n)$.
    \item If $\r\in \Comp_p(n)$ and $\_q\in \Comp_s(r_i)$, then $\r\circ_i\_q=\_t\in \Comp_{p-1+s}(n)$, with
    $$\_t=(r_1,\dots,r_{i-1},q_1,\dots,q_s,r_{i+1},\dots,r_p).$$
    \end{itemize}
\end{note}
\begin{rema}\label{partfonc}
  Equivalently, the set $\Pi(\r;n)$ can be defined as the set of maps $f:[n]\to[p]$ satisfying, for all $i\in[p]$, $\val{f^{-1}(i)}=r_i$. The identification is done as follows: a partition $R=(R_i)_{i\in[p]}$ of $[n]$ is identified with the function $\p_R:[n]\to[p]$, mapping $x\in[n]$ to the unique $i\in[p]$ such that $x\in R_i$. Conversely, each $f:[n]\to[p]$ is associated with the partition $(f^{-1}(i))_{i\in [p]}$ of $[n]$. For $\sigma\in\∑_p$ and $f:[n]\to[p]$, one has $\sigma\cdot f:x\mapsto \sigma(f(x))$.
  Following this setting, for $f\in\Pi(r_1,\dots,r_p;n)$ and $g\in\Pi(q_1,\dots,q_s;r_i)$, there exists a unique increasing bijection $b:f^{-1}(i)\to [r_i]$, so that
    \begin{eqnarray*}
      f\circ_ig:&[n]&\to[p+s-1]\\
      &x&\mapsto\left\{
      \begin{array}{lcc}
      f(x),&\mbox{if}&f(x)<i,\\
      g\circ b(x)+i-1,&\mbox{if}&f(x)=i,\\
      f(x)+s-1,&\mbox{if}&f(x)\ge i.\\
      \end{array}\right.
    \end{eqnarray*}
\end{rema}
\vskip0.5cm
\begin{defis}\textbf{Operations on partitions}\label{wedge}\label{diamond}\item
  In the next paragraphs we will define 5 operations on partitions: 1) a product $\rhd$, 2) a product $\otimes$, 3) a family of unitary operations $(\gamma_k)_{k\in\N}$, 4) a composition $\diamond$ and 5) a product $\wedge$.
\begin{enumerate}[label=\arabic*)]
  \item Let $Q\in\Pi(q_1,\dots,q_s;p)$ and $R\in \Pi(r_1,\dots,r_p;n)$. According to the previous remark, $Q$ corresponds to $\p_Q:[p]\to[s]$ and $R$ to $\p_R:[n]\to[p]$. Define the partition
  $$Q\rhd R\in\Pi\Big(\sum_{i\in Q_1}r_i,\dots,\sum_{i\in Q_s}r_i;n\Big),$$
  associated with the function $\p_Q\circ\p_R$, that is:
  $$(Q\rhd R)_{i\in[s]}=\Big(\coprod_{j\in Q_i}R_j\Big)_{i\in[s]}.$$
Note that, if $\_q\in \Comp_s(p)$ and $\r\in \Comp_p(n)$, then $\_q\rhd\r=\_t\in \Comp_s(n)$, where $\_t=\big(\sum_{i\in \_q_1}r_i,\dots,\sum_{i\in\_q_s}r_i\big)$.

  \item The operation:
  \begin{eqnarray*}
    \otimes:&\Pi(r_1,\dots,r_p;n)\times\Pi(q_1,\dots,q_s;m)&\to\Pi(r_1,\dots,r_p,q_1,\dots,q_s;n+m)\\
    &((R_i)_{i\in[p]},(Q_j)_{j\in[s]})&\mapsto R\otimes Q:=({(n,m)}\circ_2Q)\circ_1R.
  \end{eqnarray*}
satisfies $R\otimes Q=(R_1,\dots,R_p,Q_1+n,\dots,Q_s+n)$ and is associative.
  \item For $k$ a positive integer and  $R\in\Pi(r_1,\dots,r_p;n)$, $R^{\otimes k}$ denotes the partition:
$$R^{\otimes k}:=\underset{k}{\underbrace{R\otimes\dots\otimes R}}\in\Pi(\underset{k}{\underbrace{r_1,\dots,r_p,\dots,r_1,\dots,r_p}};kn).$$
  For all positive integers $k$ and $R\in\Pi(r_1,\dots,r_p;n)$, set:
  $$\gamma_k(R):=(\{i+(j-1)p\}_{j\in[k]})_{i\in[p]}\rhd R^{\otimes k}=(\coprod_{j=0}^{k-1}R_i+jn)_{i\in[p]}\in\Pi(kr_1,\dots,kr_p;kn).$$      
\begin{ex*}
  For $n=3$, $R=(\{1,3\},\{2\})$, $k=3$,
  $$R^{\otimes 3}=(\{1,3\},\{2\},\{4,6\},\{5\},\{7,9\},\{8\}),$$
  and:
  $$\gamma_3(R)=(\{1,3,4,6,7,9\},\{2,5,8\}).$$
\end{ex*}
\begin{rema*}
   $\r^{\otimes k}=\_u$, where $\_u=(r_1,\dots,r_p,\dots,r_1,\dots,r_p)$. On the other hand, it is clear that $\gamma_k(\r)$ is not the element $\_v\in\Pi(\_v;nk)$ with $\_v=(kr_1,\dots,kr_p)$. Note that $\gamma_k(c_n)=c_{nk}$ and that $\r=c_{r_1}\otimes\dots\otimes c_{r_p}$.
\end{rema*}
  \item Let $\r\in\Comp_p(n)$ and, for all $i\in[p]$, let $\q_i\in \Comp_{s_i}(m_i)$. For all $Q_1\in\Pi(\_q_1;m_1)$,\newline$Q_2\in\Pi(\_q_2;m_2)$,\dots,$Q_p\in\Pi(\_q_p;m_p)$, set:
  $$\r\diamond(Q_1,\dots,Q_p)=\gamma_{r_1}(Q_1)\otimes\dots\otimes\gamma_{r_p}(Q_p).$$
\begin{ex*}
  For $\r=(3,2)$, $\q_1=(2,1)$, $\q_2=(1,2)$,
  \begin{align*}
    \r\diamond(\q_1,\q_2)&=\gamma_3(\q_1)\otimes\gamma_2(\q_2)\\
    &=(\{1,2,4,5,7,8\},\{3,6,9\})\otimes (\{1,4\},\{2,3,5,6\})\\
    &=(\{1,2,4,5,7,8\},\{3,6,9\},\{10,13\},\{11,12,14,15\}).
  \end{align*}
\end{ex*}
The operation $\diamond$ will add a subscripted index to the partitions: for $i\in[p]$ and $j\in[s_i]$, denote by $\r\diamond(Q_1,\dots,Q_p)_{i,j}=\r\diamond(Q_1,\dots,Q_p)_{s_1+\dots+s_{i-1}+j}=r_1m_1+\dots+r_{i-1}m_{i-1}+(\gamma_{r_i}(Q_i))_j$. The previous example gives:
  $$\r\diamond(\q_1,\q_2)_{2,1}=\{10,13\}.$$
\item
  For $R\in\Pi_p(n)$ and $Q\in \Pi_s(n)$, set:
  $$(R\wedge Q)_{i\in[p],\ j\in[s]}=(R_i\cap Q_j)_{i\in[p],\ j\in[s]}.$$

    The operation $\wedge$ also adds a subscripted index to the partitions: for example, the partition $(\r\diamond(Q_1,\dots,Q_p))\wedge I$ has 3 indices. For instance:
  $$\big((\r\diamond(Q_1,\dots,Q_p))\wedge I\big)_{ijk}=(\r\diamond(Q_i))_{ij}\cap I_k.$$
  The index $i$ corresponds to the part of $\r$ and to the chosen $Q_i$. The index $j$ corresponds to the part $Q_{i,j}$ and the index $k$ corresponds to the part $I_k$.
\end{enumerate}
\end{defis}
\begin{lemm}\label{wedgecr}
  For $\r\in\Comp_p(n)$ and $R\in\Pi_s(n)$, if $\p_R$ (see \Cref{partfonc}) is increasing on each $\r_i$, then $\r\wedge R\in\Comp_{ps}(n)$.
\end{lemm}
\begin{proof}
  Let $\q=(\val{\r_1\cap R_1},\dots,\val{\r_1\cap R_s},\dots,\val{\r_p\cap R_s})\in \Comp_{ps}(n)$. We have to prove that $\r\wedge R=\q$. For all $i\in[p]$, $r_i\cap R_1$ contains $\q_{s(i-1)+1}$ elements. Hence, the fact that $R$ is increasing on $\r_1$ implies that $\{1,\dots,\q_1\}\subset R_1$, and so, that $\{1,\dots,\q_1\}\in (\r\wedge R)_{11}$. Thus, $(\r\wedge R)_{i1}=\q_{s(i-1)+1}$.  The same reasoning shows that $(\r\wedge R)_{i2}=\q_{s(i-1)+2}$, and so on, until $(\r\wedge R)_{is}=\q_{si}$.
\end{proof}
\subsection{Permutation representations of finite groups}\label{B}
This subsection aims to shed light on an isomorphism between invariants and coinvariants of a permutation representation. The following general lemma will be used in our construction of generalised divided power operations.

Let $G$ be a finite group and let $X$ be a $G$-set. 
\begin{lemm}
  There is a commutative diagram:
  $$\diag{\F[G\backslash X]\ar[r]^{\phi_G}&\F[X]^G\\ 
  \F[X]_G\ar[u]_{\psi_G}\ar[ur]_{\O_G}}$$
  where the arrows are linear isomorphisms.
\end{lemm}
\begin{proof}
  The map $\phi_G$ is induced by
  \begin{eqnarray*}
    &G\backslash X&\to\F[X]^G\\
    &[x]&\mapsto \sum_{y\in\Omega_G(x)}y,
  \end{eqnarray*}
  It is easy to show that this gives an isomorphism. It is also easy to show that the map $\F[X]\to \F[G\backslash X]$, induced by:
  \begin{eqnarray*}
    &X&\to\F[G\backslash X]\\
    &x&\mapsto[x]
  \end{eqnarray*}
  passes to the quotient by the action of $G$, giving an isomorphism $\psi_G:\F[X]_G\to\F[G\backslash X]$. Set then $\O_G=\phi_G\circ\psi_G$.
\end{proof}
\subsection{\texorpdfstring{$\∑$}{Sigma}-modules and operads}\label{modop}

This section contains the definition of a $\∑$-module, of the tensor and the composition product of $\∑$-modules, of the Schur functor associated to a $\∑$-module and the coinvariant version $\Gamma$, as well as the definition of operads, algebras over an operad and divided power algebras over a reduced operad. For general definitions concerning operads in a symmetric monoidal category $\C$, we refer to Sections 5.1 and 5.2 of \cite{LV}, and Section 1.1 of \cite{BF}.

The category of operads in $Set$ is denoted $Op_{Set}$ and objects in $Op_{Set}$ are called set operads. Similarly, $Op_\F$ denotes the category of operads in the category of $\F$-vector spaces and objects in $Op_\F$ are just called operads. 
\begin{defi}
  A $\∑$-module $\M$ is a sequence $(\M(n))_{n\in\N}$, where for each $n\in\N$, $\M(n)$ is a representation of the symmetric group $\∑_n$. A morphism between two $\∑$-modules is a sequence of equivariant maps. We denote by $\∑_{mod}$ the category of $\∑$-modules. The arity of $x\in\M(n)$ is $n$.
\end{defi}
\begin{defi}
  For $\M,\ı$ two $\∑$-modules, the $\∑$-module $\M\otimes\ı$ is given by:
  $$(\M\otimes\ı)(n):=\bigoplus_{i+j=n}\Ind_{\∑_i\times\∑_j}^{\∑_n}\M(i)\otimes\ı(j).$$
\end{defi}
The tensor product $\otimes$ makes $\∑_{mod}$ into a symmetric monoidal category with unit the $\∑$-module which is $\F$ in arity 0 and is $0$ in positive arity. Therefore, given a $\∑$-module $\M$ and an integer $n$, $\∑_n$ acts on $\M^{\otimes n}$. 
\begin{defi}
  For $\M,\ı$ two $\∑$-modules, the $\∑$-modules $\M\circ \ı$ and $\M \tilde\circ\ı$ are given by:
  $$(\M\circ\ı)(n):=\bigoplus_{k\in\N}(\M(k)\otimes \ı^{\otimes k}(n))_{\∑_k},$$
  and:
  $$(\M\tilde\circ\ı)(n):=\bigoplus_{k\in\N}(\M(k)\otimes \ı^{\otimes k}(n))^{\∑_k},$$
  where $\∑_k$ acts diagonally on $\M(k)\otimes \ı^{\otimes k}(n)$.
\end{defi}
The operations $\circ$ and $\tilde\circ$ provide as well $\∑_{mod}$ with a (non-braided) monoidal category structure, with unit, in both cases, the $\∑$-module $\scr U$ which is $\F$ in arity 1 and $0$ in any other arity.
\begin{prop}[\cite{LV}, Section 5.3.9]\label{foncop}
  The free functor $Set\to\F_{vect}$, sending $X$ to the $\F$-vector space spanned by $X$, extends to a functor $\F[\cdot]:\Op_{ens}\to \Op_\F$.
\end{prop}
\begin{defi}
  Let $\M$ be a $\∑$-module. The Schur functor associated to $\M$ is the endofunctor of the category of $\F$-vector spaces, mapping $V$ to
  $$S(\M,V):=\bigoplus_{n\in\N}(\M(n)\otimes V^{\otimes n})_{\∑_n},$$
  where $\∑_n$ acts diagonally on $\M(n)\otimes V^{\otimes n}$. This endofunctor is denoted $S(\M)$.

  Similarly, for all $\∑$-modules $\M$, there is an endofunctor of the category of $\F$-vector spaces, which maps $V$ to
    $$\Gamma(\M,V):=\bigoplus_{n\in\N}(\M(n)\otimes V^{\otimes n})^{\∑_n},$$
    where the action of $\∑_n$ is the same as before. This endofunctor is denoted $\Gamma(\M)$.
\end{defi}
\begin{prop*}[\cite{BF}, Section 1.1.11]
  Given an operad $\P$, the Schur functor $S(\P)$ is endowed with a monad structure.
\end{prop*}
\begin{defi*}
  Given an operad $\P$, a $\P$-algebra is an algebra over the monad $S(\P)$.
\end{defi*}
\begin{defi}
  A $\∑$-module $\M$ is said to be reduced (or connected) if $\M(0)=0$.
\end{defi}
\begin{prop*}[\cite{BF}, Section 1.1.18]
  Given a reduced operad $\P$, denote by $\eta$ and $\mu$ the unit and multiplication transformations of the monad $S(\P)$. The functor $\Gamma(\P)$ is endowed with a monad structure, whose unit and multiplication transformations are denoted $\tilde\eta$ and $\tilde\mu$.
\end{prop*}
\begin{prop}[See \Cref{A}]\label{propA}
Let $\P$ be a reduced operad and $V$ be an $\F$-vector space. Denote by $\mu:\P\circ\P\to\P$ the multiplication of $\P$ and by $\tilde\mu:\Gamma(\P)\circ\Gamma(\P)\to\Gamma(\P)$ the multiplication of the monad $\Gamma(\P)$. Consider an element $\t\in\Gamma(\P,\Gamma(\P,V))$ of the form
\begin{equation*}
  \t=\sum_{\sigma\in \∑_n/\∑_\r}\sigma\cdot x\otimes \sigma\cdot\bigg(\bigotimes_{i=1}^p\Big(\sum_{\sigma_i\in\∑_{m_i}/\∑_{\q_i}}\sigma_i\cdot x_i\otimes \sigma_i\cdot\_{v_i}\Big)^{\otimes r_i}\bigg),
\end{equation*}
with $\r\in\Comp_p(n)$, $x\in \P(n)^{\∑_\r}$, and for all $i\in[p]$, $\q_i\in\Comp_{k_i}(m_i)$, $x_i\in\P(m_i)^{\∑_{\q_i}}$ and $\_{v_i}\in (V^{\otimes m_i})^{\∑_{\q_i}}$. One then has
\begin{equation*}
 \tilde\mu_V(\t)=\sum_{\tau\in\∑_{M}/\prod_{i=1}^p\∑_{r_i}\wr\∑_{\q_i}}\mu(x\otimes[\tau\otimes x_1^{\otimes r_1}\otimes\dots\otimes x_p^{\otimes r_p}]_{\prod_{i}\∑_{m_i}^{\times r_i}})\otimes \tau\cdot(\_{v_1}^{\otimes r_1}\otimes\dots\otimes\_{v_p}^{\otimes r_p}),
\end{equation*}
where $M=r_1m_1+\dots+r_pm_p$, and 
$$[\tau\otimes x_1^{\otimes r_1}\otimes\dots\otimes x_p^{\otimes r_p}]_{\prod_{i}\∑_{m_i}^{\times r_i}}\in \Ind_{\prod_i\∑_{m_i}^{\times r_i}}^{\∑_M}\bigotimes_{i}\P(m_i)^{\otimes r_i}.$$
\end{prop}
\begin{defi*}
  Given a reduced operad $\P$, a divided power $\P$-algebra, or $\Gamma(\P)$-algebra, is an algebra over the monad $\Gamma(\P)$.
\end{defi*}
\begin{defi}[\cite{BF}, Section 1.1.14]\label{Tr}
  For all $\∑$-modules $\M$ and $\ı$, there is a morphism of $\∑$-modules, which is natural in $\M$ and $\ı$, called the norm map or the trace map: 
  $$Tr_{\M,\ı}:\M\circ \ı\to\M\tilde\circ\ı,$$
  inducing a natural transformation $Tr_\M:S(\M)\to\Gamma(\M)$.
\end{defi}
\begin{prop*}[\cite{BF}, Section 1.1.15]
  If $\ı$ is reduced, then $Tr_{\M,\ı}$ is an isomorphism.
\end{prop*}
\begin{prop*}[\cite{BF}, Section 1.1.19]
  If $\P$ is a reduced operad, then $Tr_\P$ is a morphism of monads.
\end{prop*}
\begin{prop*}[\cite{BF}, Section 1.1.1]
  If $\F$ has characteristic 0, then $Tr_\P$ is an isomorphism.
\end{prop*}
Given a reduced operad $\P$, the composition product of $\Gamma(\P)$ involves the isomorphism $Tr_{\P,\P}$ (see \cite{BF}, Section 1.1.18 and \Cref{A}). According to the foregoing, a $\Gamma(\P)$-algebra is also endowed with a structure of a $\P$-algebra and, if the characteristic of the ground field is 0, then those two structures coincide.

\section{Operations indexed by invariant elements}\label{secinv}

In this section, we describe the category of $\Gamma(\P)$-algebras associated to a reduced operad $\P$ in terms of operations parametrised by invariants elements of $\P(n)$ under the action of Young subgroups of $\∑_n$. These polynomial operations are to be seen as generalised divided power operations. 

Formally, the goal of this section is to explain the definition of comparison functors that fit in a diagram of the following form:
$$\diag{
  \Gamma(\M)_{mod}\ar@/_/[d]_{F_\M}&\Gamma(\P)_{alg}\ar@{_{(}->}[l]\ar@/_/[d]_{F_\P}\\
  \beta(\M)\ar@/_/[u]_{G_\M}&\beta'(\P)\ar@{_{(}->}[l]\ar@/_/[u]_{G_\P}},
$$
where $\beta(\M)$ and $\beta'(\P)$ will be categories of vector spaces endowed with generalised divided power operations, and to check that the vertical functors define isomorphisms of categories.

In \Cref{subsecF}, we explain the definition of the categories $\beta(\M)$ and $\beta'(\M)$ with full details and the definition of the functors $F_\M$ and $F_\P$.
The functor $F_\P$ endows any $\Gamma(\P)$-algebra with a generalised divided power algebra structure.

In \Cref{subseclemminv}, we explain the definition of the functors $G_\M$ and $G_\P$ and we check that these functors are inverse to $F_\M$ and $F_\P$ in order to complete the proof of our statement.

In \Cref{secCOM}, we examine the applications of our constructions to the case of the commutative operad.

The main result of this section is \Cref{theoinv}, stated in the introduction of this article as \Cref{theoinvintro}.


\subsection{Construction of the functors \texorpdfstring{$F_\M$}{} and \texorpdfstring{$F_\P$}{}}\label{subsecF}
In this subsection, we define categories $\Gamma(\M)_{mod}$ and $\Gamma(\P)_{alg}$ of divided $\M$-modules and divided $\P$-algebras, as well as categories $\beta(\M)$ and $\beta'(\P)$, which have objects vector spaces endowed with polynomial operations $\beta_{x,\r}$ indexed by compositions $\r\in\Comp_p(n)$ of integers and invariant elements of $\P$, $x\in\P(n)^{\∑_\r}$. One of the main results of this article is the \cref{theoinv}, which states that $\Gamma(\P)_{alg}$ and $\beta'(\P)$ are isomorphic, and therefore, that the data of a divided power algebra over an operad is equivalent to the data of generalised divided power operations on the underlying vector space.

In this section, we will need the notation defined in \Cref{Partitions} to deal with compositions of integer and partitions of the set $[n]$. 

Let $\M$ and $\P$ be a $\∑$-module and a reduced operad.

\begin{defi}\label{defbeta}
\begin{itemize}
  \item The category $\Gamma(\M)_{mod}$ has objects $(A,f)$ where $A$ is an $\F$-vector space and $f$ is a linear morphism $f:\Gamma(\M,A)\to A$. A morphism of $\Gamma(\M)_{mod}$ of the form $(A,f)\to (A',f')$ is a linear map $g:A\to A'$ such that $f'\circ\Gamma(\M,g)=g\circ f$.
  \item The category $\beta(\M)$ has objects $(A,\beta)$, where $A$ is an $\F$-vector space and $\beta_{*,*}$ is a family of operations $\beta_{x,\r}:A^{\times p}\to A$, given for all $\r\in\Comp_p(n)$ and $x\in\M(n)^{\∑_\r}$, and which satisfy the relations:
  \begin{enumerate}[label=($\beta$\arabic*),itemsep=5pt]
    \item\label{relperm}$\beta_{x,\r}((a_i)_i)=\beta_{\rho^*\cdot x,\r^\rho}((a_{\rho^{-1}(i)})_{i})$ for all $\rho\in\∑_p$, where $\rho^*$ denotes the block permutation with blocks of size $(r_i)$ associated to $\rho$.

    \item\label{rel0} $\beta_{x,(0,r_1,r_2,\dots,r_p)}(a_0,a_1,\dots,a_p)=\beta_{x,(r_1,r_2,\dots,r_p)}(a_1,\dots,a_p)$.

    \item\label{rellambda} $\beta_{x,\r}(\lambda a_1,a_2,\dots,a_p)=\lambda^{r_1}\beta_{x,\r}(a_1,\dots,a_p)\quad\forall \lambda\in \F$.

    \item\label{relrepet}
    If $\r\in \Comp_p(n)$ and $\q\in \Comp_s(p)$, then
    $$\beta_{x,\r}(\undercount{q_1}{a_1,\dots,a_1},\undercount{q_2}{a_2,\dots,a_2},\dots,\undercount{q_s}{a_s,\dots,a_s})=\beta_{\big(\sum_{\sigma\in \∑_{\q\rhd\r}/\∑_\r}\sigma\cdot x\big),\ \q\rhd\r}(a_1,a_2,\dots,a_s).$$

    \item\label{relsomme} $\beta_{x,\r}(a_0+a_1,\dots,a_p)=\sum_{l+m=r_1}\beta_{x,\r\circ_1(l,m)}(a_0,a_1,\dots,a_p)$.

    \item\label{rellin} $\beta_{\lambda x+y,\r}=\lambda\beta_{x,\r}+\beta_{y,\r}$ , for all $x,y\in\M(n)^{\∑_\r}$.
  \end{enumerate}
  The morphisms are the linear maps compatible with the operations $\beta_{x,\r}$.
\end{itemize}
If $(A,f)$ is an object of $\Gamma(\M)_{mod}$, or if $(A,\beta)$ is an object of $\beta(\M)$, then the vector space $A$ is the underlying vector space of the object.
\end{defi}
\begin{rema}\label{rempart}  
If $(A,\beta)$ is an object in the category $\beta(\M)$, the operations $\beta_{*,*}$ induce operations $\beta_{x,R}:A^{\times p}\to A$, for all partitions $R\in\Pi(r_1,\dots,r_p;n)$ and all $x$ stable under the action of $\∑_R$. Indeed, according to \Cref{transn}, there exists $\tau\in\∑_n$ such that $\tau(R)=\r$. Generally, the permutation $\tau\in\∑_n$ such that $\tau(R)=\r$ is not unique. One can define:
$$\beta_{x,R}(a_1,\dots,a_p):=\beta_{\tau\cdot x,\r}(a_1,\dots,a_p),$$
which does not depend on the chosen $\tau$.
\end{rema}
\begin{rema}\label{remacting}
  The relations \ref{rel0}, \ref{rellambda} and \ref{relsomme} are expressed here by acting on the first variable, but can be rewritten, thanks to the relation \ref{relperm}, to act on any other variable. For example, from the relation \ref{relsomme}, one deduces the equivalent relation:
  \begin{enumerate}[label=($\beta$\arabic* bis)]
    \setcounter{enumi}{4}
    \item \label{relsomme+}For all $\r\in \Comp_p(n)$,
     and $q\in \Comp_p(m)$,
    $$\beta_{x,\r}\Big(\sum_{i\in\_q_1}a_i,\dots,\sum_{i\in\_q_p}a_i\Big)=\sum_{\_k_1\in \Comp_{q_1}(r_1),\dots,\_k_p\in \Comp_{q_p}(r_p)}\beta_{x,\r\circ(\_k_1,\dots,\_k_p)}(a_1,\dots,a_m).$$
  \end{enumerate}
\end{rema}
\begin{defi}\label{defbeta'}
  \begin{itemize}
    \item The category $\Gamma(\P)_{alg}$ is the category of algebras over the monad $\Gamma(\P)$.
    \item The category $\beta'(\P)$ is the full subcategory of objects $(A,\beta)$ of $\beta(\P)$ satisfying:
    \begin{enumerate}[label=($\beta$7\alph*)]
    \item \label{relunit}$\beta_{id_\P,(1)}(a)=a \quad \forall a\in A$.
    \item \label{relcomp} Let $\r\in \Comp_p(n)$, $x\in \P(n)^{\∑_\r}$ and for all $i\in[p]$, let $\q_i\in \Comp_{k_i}(m_i)$, $x_i\in\P(m_i)^{\∑_{\q_i}}$ and $(b_{ij})_{1\le j\le k_i}\in A^{\times k_i}$. Denote by $b=(b_{ij})_{i\in[p],j\in[k_i]}$ and for all $i\in[p]$, $b_i=(b_{i,j})_{j\in[k_i]}$. Then:
    $$\beta_{x,\r}(\beta_{x_1,\q_1}(b_{1}),\dots,\beta_{x_p,\q_p}(b_{p}))=\beta_{\sum_{\tau}\tau\cdot\mu \big(x\otimes \big(\bigotimes_{i=1}^px_i^{\otimes r_i}\big)\big),\r\diamond(\q_i)_{i\in[p]}}(b),$$
    where $\r\diamond(\q_i)_{i\in[p]}$ is defined in \Cref{diamond}, where $\beta_{\cdot,\r\diamond(\q_i)_{i\in[p]}}$ is defined in \Cref{rempart} and where $\tau$ ranges over $\∑_{\r\diamond(q_i)_{i\in[p]}}/(\prod_{i=1}^{p}\∑_{r_i}\wr \∑_{\q_i})$ in the sum.
  \end{enumerate}
  \end{itemize}
\end{defi}
\begin{prop}\label{prepM}
 Let $(A,f)$ be an object of $\Gamma(\M)_{mod}$. For all $n\in\N$, all compositions $\r\in\Comp_p(n)$ and all $x\in\M(n)^{\∑_\r}$, define:
\begin{eqnarray*}
  \beta_{x,\r}:&A^{\times p}&\to A\\
  &(a_{i})_{1\le i\le p}&\mapsto f\Big(\sum_{\sigma\in\∑_n/\∑_\r}\sigma\cdot x\otimes\sigma\cdot(\otimes_ia_i^{\otimes r_i})\Big).
\end{eqnarray*}
The mapping $(A,f)\mapsto(A,\beta)$ extends to a functor $F_\M:\Gamma(\M)_{mod}\to\beta(\M)$. In particular, the operations $\beta_{*,*}$ satisfy relations \ref{relperm} to \ref{rellin} of \Cref{defbeta}.  
\end{prop}
\begin{prop}\label{prepP}
  Let $\P$ be a reduced operad and $(A,f)$ be an object of $\Gamma(\P)_{alg}$. Then $(A,\beta)$, as defined in \Cref{prepM}, is an object of $\beta'(\P)$. Hence, $F_\M$ restricts to a functor $F_\P:\Gamma(\P)_{alg}\to\beta'(\P)$.
\end{prop}
\begin{proof}[Proof of \Cref{prepM} and \Cref{prepP}]
  Relations \ref{relperm}, \ref{rel0}, \ref{rellambda}, \ref{rellin} and \ref{relunit} are clear.

  Relation \ref{relrepet} comes from the following fact:
    \begin{equation*}
      \sum_{\sigma\in\∑_{n}/\∑_{\r}}\sigma\cdot x=\sum_{\sigma\in\∑_n/\∑_{\q\rhd\r}}\sigma\cdot\sum_{\tau\in\∑_{\q\rhd\r}/\∑_\r}\tau\cdot x.
    \end{equation*}
  
  Relation \ref{relsomme} comes from the following fact:
    \begin{equation*}
      (a+b)^{\otimes u}=\sum_{(l,m)\in \Comp_2(u)}\sum_{\sigma\in\∑_{u}/\∑_{l}\times\∑_m}\sigma\cdot (a^{\otimes l}\otimes b^{\otimes m}).
    \end{equation*}

  Relation \ref{relcomp} makes sense, firstly because $\prod_{i=1}^p\∑_{r_i}\wr \∑_{\q_i}\subseteq\∑_{\r\diamond(\q_i)_{i\in[p]}}$, and secondly because $\mu \bigg(x\otimes\Big[\sum_{\tau\in\∑_{\r\diamond(\q_i)_{i\in[p]}}/\prod_{i=1}^p\∑_{r_i}\wr \∑_{\q_i}}\tau\otimes\big(\bigotimes_{i=1}^px_{i}^{\otimes r_{i}}\big)\Big]\bigg)$ is stable under the action of $\∑_{\r\diamond(\q_i)_{i\in[p]}}$. 

  It comes from the fact that, $f$ being compatible with the multiplication $\tilde\mu$ of $\Gamma(\P)$, $f\circ \Gamma(\P,f)=f\circ\tilde{\mu}_A$, where $\tilde\mu_A:\Gamma(\P,\Gamma(\P,A))\to\Gamma(\P,A)$ is the component of $\tilde\mu$ at $A$, and hence, following \Cref{A}:
    \begin{multline*}
      \beta_{x,\r}(\beta_{x_1,\q_1}(b_{1}),\dots,\beta_{x_p,\q_p}(b_{p}))=\\
      f\bigg(\sum_{\sigma\in\∑_n/\∑_\r}\sigma\cdot x\otimes \sigma \cdot\Big(\bigotimes_{i=1}^p f\Big(\sum_{\rho_i\in\∑_{m_i}/\∑_{\q_i}}\rho_i\cdot x_i\otimes \rho_i\cdot(b_{i1}^{\otimes q_{i1}}\otimes\dots\otimes b_{ik_i}^{\otimes q_{ik_i}})\Big)^{\otimes r_i}\Big)\bigg)\\
      = f\bigg(\tilde{\mu}_A\Big(\sum_{\sigma\in\∑_n/\prod_i\∑_{r_i}}\sigma\cdot x\otimes \sigma \cdot\Big(\bigotimes_{i=1}^{p}\big(\sum_{\rho_i\in\∑_{m_i}/\∑_{\q_i}}\rho_i\cdot x_i\otimes\rho_i\cdot (b_{i1}^{\otimes q_{i1}}\otimes\dots\otimes b_{ik_i}^{\otimes q_{ik_i}})\big)^{\otimes r_i}\Big)\Big)\bigg)\\
      =f\bigg(\sum_{\sigma\in \∑_{\sum_ir_im_i}/\prod_{i=1}^p\∑_{r_i}\wr \∑_{\q_i}}\sigma\cdot\mu\Big(x\otimes\big[\bigotimes_{i=1}^p x_i^{\otimes r_i}\big]\Big)\otimes \sigma\cdot\bigotimes_{i=1}^p(b_{i1}^{\otimes q_{i1}}\otimes\dots\otimes b_{ik_i}^{\otimes q_{ik_i}})^{\otimes r_i}\bigg),
    \end{multline*}\label{7bproof}
    \begin{multline*}
      =f\bigg(\sum_{\sigma\in \∑_{\sum_ir_in_i}/\∑_{\r\diamond(\q_i)_{i\in[p]}}}\sum_{\tau\in\∑_{\r\diamond(\q_i)_{i\in[p]}}/(\prod_{i=1}^{p}\∑_{r_i}\wr\∑_{\q_i})}\sigma\tau\cdot\mu\Big(x\otimes\big[\bigotimes_{i=1}^px_i^{\otimes r_i}\big]\Big)\\
      \otimes \sigma\cdot\tau\cdot\Big(\bigotimes_{i=1}^p(b_{i1}^{\otimes q_{i1}}\otimes\dots\otimes b_{ik_i}^{\otimes q_{ik_i}})^{\otimes r_i}\Big)\bigg)
    \end{multline*}
    \begin{equation*}
      =\beta_{\sum_{\tau}\tau\cdot\mu \Big(x\otimes \big(\bigotimes_{i=1}^px_i^{\otimes r_i}\big)\Big),\r\diamond(\q_i)_{i\in[p]}}(b_{ij})_{i\in[p],j\in[k_i]}.
   \qedhere \end{equation*}
\end{proof}
\begin{prop}\label{propinv}
  For any $\∑$-module $\M$, the functor $F_\M:\Gamma(\M)_{mod}\to\beta(\M)$ is an isomorphism of categories.
\end{prop}
\begin{theo}\label{theoinv}
  For any reduced operad $\P$, the functor $F_\P:\Gamma(\P)_{alg}\to\beta'(\P)$ is an isomorphism of categories.
\end{theo}
\vspace{.5cm}
\begin{proof}[Plan of the proof of \Cref{propinv} and \Cref{theoinv}]
We want to construct the functor $G_\M:\beta(\M)\to\Gamma(\M)_{mod}$, inverse of $F_\M$. Let $(A,\alpha)$ be an object of $\beta(\M)$.  \Cref{decompsum} gives a decomposition of $\Gamma(\M,A)$ as a direct sum. \Cref{deff} shows how to define a function $f:\Gamma(\M,A)\to A$ from this decomposition, which is natural in $(A,\alpha)$. We thereby obtain a functor $G_\M:\beta(\M)\to\Gamma(\M)_{mod}$, and we readily check that $G_\M\circ F_\M=id_{\Gamma(\M)_{mod}}$ (\Cref{remf}). The \Cref{lemmfin} shows that $F_\M\circ G_\M$ is the identity of $\beta(\M)$.

To prove \Cref{theoinv}, let us furthermore assume that $\M=\P$ is endowed with a structure of reduced operad and that the operations $\alpha_{x,\r}$ satisfy relations \ref{relunit} and \ref{relcomp}. It is then enough to show that the morphism $f$ of \Cref{deff} endows $A$ with the structure of a $\Gamma(\P)$-algebra, which implies that $G_\M$ restricts to a functor $G_\P:\beta'(\P)\to\Gamma(\P)_{alg}$, which will be inverse to $F_\P$. This is done in Lemmas \ref{lemmunit} and \ref{lemmcomp}.
\end{proof}

\subsection{Technical lemmas}\label{subseclemminv}
In this section, we give the construction of the functors $G_\M$ and $G_\P$, and we check that these functors are inverse to $F_\M$ and $F_\P$ following the argument given in the plan of the proof of \Cref{propinv} and \Cref{theoinv}. \Cref{decompsum} gives a decomposition of $\Gamma(\M,A)$ as a direct sum that depends on a choice of a totally ordrered basis of $A$, \Cref{deff} gives the definition of a morphism $f:\Gamma(\M,A)\to A$ from that decomposition and the definition of $G_\M$. Lemmas \ref{lemmperm} to \ref{lemmfin} show that $f$, and so $G_\M$, do not depend on the choice of the basis of $A$, and that $G_\M:\beta(\M)\to\Gamma(\M)_{mod}$ is an inverse to $F_\M$. Finally, \ref{lemmunit} and \ref{lemmcomp} show that, in the case where $\M=\P$ is endowed with a structure of operad, $G_\M$ restricts to a functor $G_\P:\beta'(\P)\to\Gamma(\P)_{alg}$ that is the inverse to $F_\P$.

\vspace{.5cm}
Let $(A,\alpha)$ be an object of $\beta(\M)$, $B$ be a totally ordered $\F$-vector space basis of $A$.
\begin{lemm}\label{decompsum}
  The module $\Gamma(\M,A)$ is isomorphic to the direct sum:
\begin{equation*}
  \bigoplus_{n,p,\r,b_1,\dots,b_p}\M(n)^{\∑_\r},
\end{equation*} 
where $n$ ranges over $\N$, $p$ ranges over $[n]$, $\r$ ranges over the set $\Comp'_p(n):=\{\r\in\Comp_p(n):\forall i\in[p], r_i>0\}$, and $b_1<\dots<b_p\in B$. Moreover, the isomorphism is given by the direct sum over $n\in\N$, $p\in[n]$, $\r\in\Comp'_p(n)$, and $b_1<\dots<b_p\in B$ of the morphisms:
\begin{eqnarray*}
  &M(n)^{\∑_\r}&\to\Gamma(\M,A)\\
  &x&\mapsto \sum_{\sigma\in\∑_n/\∑_\r} \sigma\cdot x\otimes \sigma\cdot\Big(\bigotimes_{i\in[p]}b_i^{\otimes r_i}\Big).
\end{eqnarray*}
\end{lemm}
\begin{proof}
  For all $n\in\N$, we can identify the $\∑_n$-module $A^{\otimes n}$ with:
  $$\bigoplus_{p,\r,b_1,\dots,b_p} Ind_{\∑_\r}^{\∑_n}T_\r(b_1^{\otimes r_1}\otimes\dots \otimes b_p^{\otimes r_p}),$$
 where $p$ ranges over $[n]$, $\r$ ranges over $\Comp'_p(n)$, $b_1<\dots<b_p\in B$, $Ind_{\∑_\r}^{\∑_n}$ designates the induced representation, and $T_{\r}(b_1^{\otimes r_1}\otimes\dots \otimes b_p^{\otimes r_p})$ is the trivial $\∑_{\r}$-module spanned by $b_1^{r_1}\otimes\dots\otimes b_p^{\otimes r_p}$. Remark that $Ind_{\∑_\r}^{\∑_n}T_{\r}(b_1^{\otimes r_1}\otimes\dots \otimes b_p^{\otimes r_p})$ is equal, as a $\∑_n$-module, to $\F[\∑_n/\∑_\r]\otimes T_{(n)}(b_1^{\otimes r_1}\otimes\dots \otimes b_p^{\otimes r_p})$, where $\∑_n$ acts on $\∑_n/\∑_\r$ by translation. This $\∑_n$-module is a permutation representation with basis $\{\sigma\otimes (b_1^{\otimes r_1}\otimes\dots \otimes b_p^{\otimes r_p}):\sigma\in\∑_n/\∑_\r\}$. 
 The identification $\bigoplus_{p,\r,b_1,\dots,b_p} Ind_{\∑_\r}^{\∑_n}T_{\r}(b_1^{\otimes r_1}\otimes\dots \otimes b_p^{\otimes r_p})\cong A^{\otimes n}$ is given by $\sigma\otimes (b_1^{\otimes r_1}\otimes\dots \otimes b_p^{\otimes r_p})=\sigma\cdot (b_1^{\otimes r_1}\otimes\dots \otimes b_p^{\otimes r_p})$.

  The module $(\M(n)\otimes A^{\otimes n})^{\∑_n}$ decomposes as $\bigoplus_{p,\r,b_1,\dots,b_p}\Big(M(n)\otimes Ind_{\∑_\r}^{\∑_n}T_{\r}(b_1,\dots,b_p)\Big)^{\∑_n}$. Let $p\in[n], \r\in\Comp'_p(n)$, and $b_1<\dots<b_p\in B$. Given an element $\t\in M(n)\otimes Ind_{\∑_\r}^{\∑_n}T_{\r}(b_1,\dots,b_p)$, there is a unique $(x_{\sigma})_{\sigma\in\∑_n/\∑_\r}\in M(n)^{\times \val{\∑_n/\∑_\r}}$ such that
 $$\t=\sum_{\sigma\in\∑_n/\∑_\r}x_\sigma\otimes \sigma\cdot (b_1^{\otimes r_1}\otimes\dots \otimes b_p^{\otimes r_p}).$$
 Suppose that $\t$ is invariant under the action of $\∑_n$. Then, for all $\tau\in\∑_n$, the uniqueness of $(x_\sigma)_{\sigma\in \∑_n/\∑_\r}$ implies that $\tau\cdot x_\sigma=x_{\tau\cdot \sigma}$, and in particular, if $\tau\in \∑_\r$, $\tau\cdot x_\sigma=x_\sigma$. Finally, $\t=\sum_{\sigma\in\∑_n/\∑_\r}\sigma\cdot x_{id}\otimes \sigma\cdot\Big(\bigotimes_{i}b_{i}^{\otimes r_{i}}\Big)$, where $id\in \∑_n/\∑_\r$ is the class of the neutral element. Hence, an element $x\in\M(n)^{\∑_\r}$ uniquely determines the element
\begin{equation*}
  \sum_{\sigma\in\∑_n/\∑_\r}\sigma\cdot x\otimes \sigma\cdot\Big(\bigotimes_{i}b_{i}^{\otimes r_{i}}\Big)\in \Big(M(n)\otimes Ind_{\∑_\r}^{\∑_n}T_{\r}(b_1,\dots,b_p)\Big)^{\∑_n}.\qedhere
\end{equation*}
\end{proof}
\begin{defi}\label{deff}
	Given $(A,\alpha)$ an object of $\beta(\M)$, there is a linear map
	\begin{eqnarray*}
		f:&\Gamma(\M,A)&\to A\\
    &\t=\sum_{\sigma\in\∑_{n}/\∑_\r}\sigma\cdot x\otimes \sigma\cdot\Big(\bigotimes_{i}b_{i}^{\otimes r_{i}}\Big)&\mapsto \alpha_{x,\r}(b_1,\dots,b_p)\ .
	\end{eqnarray*}
	Denote by $G_\M(A,\alpha):=(A,f)$.
\end{defi}
\begin{rema}\label{remf}
	Indeed, the expression of $f$ is linear in $x\in\M(n)^{\∑_\r}$. With this definition, it is clear that if $G_\M(F_\M(A,g))=G_\M(A,\alpha)=(A,f)$, then $f=g$, for $f$ and $g$ are linear and are equal on each factor of the decomposition of $\Gamma(\M,A)$ as a direct sum given by \Cref{decompsum}.
\end{rema}
To prove that this extends to a functor $G_\M:\beta(\M)\to\Gamma(\M)_{mod}$ which is inverse to $F_\M$, one has to show that $F_\M\circ G_\M$ is the identity of $\beta(\M)$ and that $G_\M$ is actually a functor, that is, the construction of $f$ does not depend on the basis $B$ of $A$ used. For that, let $(A,\beta)=F_\M(A,f)$ where $f$ is the morphism of \Cref{deff}. The next lemmas show that $\alpha_{x,\r}=\beta_{x,\r}$ for all $\r\in\Comp_p(n)$ and $x\in\M(n)^{\∑_\r}$.

It is clear that, for all $\r\in\Comp_p(n)$, all $b_1<\dots<b_p\in B$ and all $x\in\M(n)^{\∑_\r}$,
\begin{equation}\label{eq0}
  \beta_{x,\r}(b_1,\dots,b_p)=\alpha_{x,\r}(b_1,\dots,b_p)=f(\t)
\end{equation}
\begin{lemm}\label{lemmperm}
	 For $\r\in\Comp_p(n)$, for $(b_i)_{1\le i\le p}\in B^{\times p}$ pairwise distinct and for $x\in \M(n)^{\∑_\r}$, we have the equality:
    \begin{equation}\label{eq1}
    \beta_{x,\r}(b_i)_i=\alpha_{x,\r}(b_i)_i.
    \end{equation}
\end{lemm}
\begin{proof}
	There exists a permutation $\rho\in\∑_p$ such that $\rho\cdot(b_i)_i$ is ordered, and:
	\begin{align*}
		\beta_{x,\r}(b_i)_i\overeg{Relation \ref{relperm}}\beta_{\rho^*\cdot x, \r^\rho}(b_{\rho^{-1}(i)})_i
		&\overeg{\Cref{eq0}}\alpha_{\rho^*\cdot x,\r^\rho}(b_{\rho^{-1}(i)})_i
		\\&\overeg{Relation \ref{relperm}}\alpha_{x,\r}(b_i)_i.\qedhere
	\end{align*}
\end{proof}
\begin{lemm}\label{lemmrepet}
	Equality (\ref{eq1}) holds for $\r\in\Comp_p(n)$, for $(b_i)_{1\le i\le p}\in B^{\times p}$ and for $x\in \M(n)^{\∑_\r}$.
\end{lemm}
\begin{proof}
	According to \Cref{lemmperm}, $(b_i)_i$ can be sorted in order to obtain only consecutive repetitions. Yet, if $\q\in \Comp_s(p)$ and if $(b_1,\dots,b_s)\in B^{\times s}$ are pairwise distinct, then:
	\begin{align*}
		\beta_{x,\r}(\underset{q_1}{\underbrace{b_1,\dots,b_1}},\dots,\underset{q_s}{\underbrace{b_s,\dots,b_s}})&\overset{\mbox{Relation \ref{relrepet}}}{=}\beta_{\sum_{\sigma\in\∑_{\q\rhd\r}/\∑_\r}\sigma\cdot x,\q\rhd\r}(b_1,b_2,\dots,b_p)\\
		&\overset{\mbox{\Cref{lemmperm}}}{=}\alpha_{\sum_{\sigma\in\∑_{\q\rhd\r}/\∑_\r}\sigma\cdot x,\q\rhd\r}(b_1,b_2,\dots,b_p)\\
		&\overset{\mbox{Relation \ref{relrepet}}}{=}\alpha_{x,(r_1,\dots,r_p)}(\underset{q_1}{\underbrace{b_1,\dots,b_1}},\dots,\underset{q_s}{\underbrace{b_s,\dots,b_s}}).\qedhere\\
	\end{align*}
\end{proof}
\begin{lemm}\label{lemmfin}
	Equality (\ref{eq1}) holds for $\r\in\Comp_p(n)$, $(b_1,\dots,b_p)\in A^{\times p}$ and $x\in \M(n)^{\∑_\r}$. The definition of $f$ does not depend on the choice of $B$, and $ F_\M\circ G_\M=id_{\beta(\M)}$.
\end{lemm}
\begin{proof}
	For $i\in[p]$, let us write $b_i$ in the basis $B$ as
	$$b_i=\sum_{j=1}^{k_i}\lambda_{ij}x_{ij}.$$
	Set $X=(x_{ij})_{i\in[p],j\in[k_i]}$ and $\Lambda X=(\lambda_{ij}x_{i,j})_{i\in[p],j\in[k_i]}$. One has:
		\begin{align*}
		\beta_{x,\r}(b_1,\dots,b_p)&\overeg{Relation \ref{relsomme+}}\sum_{\q_1\in \Comp_{k_1}(r_1),\dots,\q_p\in \Comp_{k_p}(r_p)}\beta_{x,\r\circ(\q_1,\dots,\q_p)}(\Lambda X)\\
		&\overeg{Relation \ref{rellambda}}\sum_{\q_1\in \Comp_{k_1}(r_1),\dots,\q_p\in \Comp_{k_p}(r_p)}\bigg(\prod_{i=1}^p\prod_{j=1}^{k_i}\lambda_{ij}^{q_{ij}}\bigg)\beta_{x,\r\circ(\q_1,\dots,\q_p)}(X)\\
		&\overeg{Lemme \ref{lemmrepet}}\sum_{\q_1\in \Comp_{k_1}(r_1),\dots,\q_p\in \Comp_{k_p}(r_p)}\bigg(\prod_{i=1}^p\prod_{j=1}^{k_i}\lambda_{ij}^{q_{ij}}\bigg)\alpha_{x,\r\circ(\q_1,\dots,\q_p)}(X)\\
		&\overeg{Relations \ref{rellambda} et \ref{relsomme+}}\alpha_{x,\r}(b_1,\dots,b_p).\qedhere
	\end{align*}
\end{proof}
According to \Cref{lemmfin}, the map $f$ as defined does not depend on the choice of $B$ and the resulting functor $G_\M$ is indeed an inverse to $F_\M$. Hence, \Cref{propinv} is proven. To prove \Cref{theoinv}, let us assume that $\M=\P$ is endowed with the structure of a reduced operad, and let us show that $G_\M$ restricts into a functor $\beta'(\P)\to\Gamma(\P)_{alg}$. Denote by $\eta$ and $\mu$ the unit and multiplication of $S(\P)$, and by $\tilde\eta$ and $\tilde\mu$ those of $\Gamma(\P)$ (\Cref{A} gives the full expression of $\tilde \mu$).
\begin{lemm}\label{lemmunit}
	The map $f$ is compatible with $\tilde\eta$.
\end{lemm}
\begin{proof}
Indeed, using Relation \ref{relunit}:
\begin{equation*}
	f\circ\tilde\eta(a)=f(id_\P\otimes a)
	=\alpha_{id_\P,(1)}(a)
	=a.\qedhere
\end{equation*}
\end{proof}
\begin{lemm}\label{lemmcomp}
	The map $f$ is compatible with $\tilde{\mu}$.
\end{lemm}
\begin{proof}
Let us check that the following diagram commutes:
\begin{equation*}
	\diag{\Gamma(\P,\Gamma(\P,A))\ar[rr]^{\Gamma(\P,f)}\ar[d]_{\tilde{\mu}_A}&&\Gamma(\P,A)\ar[d]^{f}\\
	\Gamma(\P,A)\ar[rr]_f&&A}
\end{equation*}
A generic element of $\Gamma(\P,\Gamma(\P,A))$ is of the form:
\begin{equation*}
	\t:=\sum_{\sigma\in\∑_{n}/\∑_\r}\sigma\cdot x\otimes \sigma\cdot\Big(\bigotimes_{i=1}^p\t_i^{\otimes r_i}\Big),
\end{equation*}
where $\∑_\r\subseteq Stab(x)$ and where, for all $i\in[p]$,
\begin{equation*}
	\t_i:=\sum_{\sigma\in\∑_{n_i}/\∑_{\q_i}}\sigma\cdot x_{i}\otimes\sigma\cdot(\bigotimes_{j=1}^{s_i}a_{ij}^{\otimes q_{ij}}),
\end{equation*}
with $\q_i\in \Comp_{s_i}(n_i)$ and $\∑_{\q_i}\subseteq Stab(x_i)$.

Set $a:=(a_{ij})_{i\in[p],j\in[s_i]}$ and, for all $i\in[p]$, $a_i:=(a_{ij})_{j\in[s_i]}$. On one hand:
\begin{align*}
	f\circ\Gamma(\P,f)(\t)&=f\bigg(\sum_{\sigma\in\∑_{n}/\∑_\r}\sigma\cdot x\otimes \sigma\cdot\Big(\bigotimes_{i}f(\t_{i})^{\otimes r_{i}}\Big)\bigg)\\
	&=\alpha_{x,r}(\alpha_{x_{1},\q_1}(a_{1}),\dots,\alpha_{x_{p},\q_p}(a_{p})),
\end{align*}
which is equal, according to relation \ref{relcomp}, to:
\begin{equation*}
	\alpha_{\sum_{\tau\in G}\tau\cdot\mu \big(x\otimes\big[\bigotimes_{i}x_{i}^{\otimes r_{i}}\big]\big),\r\diamond(\q_1,\dots,\q_s)}(a).
\end{equation*}
On the other hand, setting $H:=\∑_{\r\diamond(\q_1,\dots,\q_p)}$, one has
\begin{equation*}
	\hspace{-39.61874pt}f\circ\tilde{\mu}_A(\t)=f\bigg(\sum_{(\sigma,\tau)\in \Big(\∑_{\sum_ir_in_i}/H\Big)\times \Big(H/\prod_i\∑_{r_i}\wr\big(\prod_j\∑_{q_{ij}}\big)\Big)}\sigma\tau\cdot\Big(\mu\Big(x\otimes\Big[\bigotimes_{i=1}^px_i^{\otimes r_i}\Big]\Big)\Big)\otimes \sigma\cdot\Big(\bigotimes_{i=1}^p\Big(\bigotimes_{j=1}^{s_i}a_{ij}^{\otimes q_{ij}}\Big)^{\otimes r_i}\Big)\bigg)
\end{equation*}
Following the same computation as in the proof of relation \ref{relcomp} of \Cref{prepP} (page \pageref{7bproof}), one gets
\begin{equation*}
	f\circ\tilde{\mu}_A(\t)=\alpha_{\sum_{\tau\in G}\tau\cdot\mu \big(x\otimes\big[\bigotimes_{i}x_{i}^{\otimes r_{i}}\big]\big),\r\diamond(\q_1,\dots,\q_s)}(a).\qedhere
\end{equation*}
\end{proof}
\subsection{Example: the operad \texorpdfstring{$Com$}{Com}}\label{secCOM}
In this section, we check how our description of the structure of a $\Gamma(\P)$-algebra applies in the particular case $\P = Com$. We retrieve the classical definition of divided power operations on commutative algebras, as given by H. Cartan in \cite{HC}.

\begin{defi}\label{puissdiv}
A divided power algebra is an associative and commutative algebra $A$ equipped, for every integer $n\in\N$, of an operation $\gamma_n:A\to A$ satisfying:
\begin{enumerate}[label=(C\arabic*),itemsep=5pt]
\item\label{relComlambda} $\gamma_n(\lambda a)=\lambda^n\gamma_n(a)$ for all $a\in A_i$ et $\lambda\in\F$
\item\label{relComrepet} $\gamma_m(a)\gamma_n(a)=\binom{m+n}{m}\gamma_{m+n}(a)$ for all $a\in A$
\item\label{relComsomme} $\gamma_n(a+b)=\gamma_n(a)+\big(\sum_{l=1}^{n-1}\gamma_l(a)\gamma_{n-l}(b)\big)+\gamma_n(b)$ for all $a\in A$, $b\in A$
\item\label{relComunit} $\gamma_1(a)=a$ for all $a\in A$
\item\label{relComcomp1} $\gamma_n(ab)=n!\gamma_n(a)\gamma_n(b)=a^n\gamma_n(b)=\gamma_n(a)b^n$ for all $a\in A$, $b\in A$.
\item\label{relComcomp2} $\gamma_m(\gamma_n(a))=\frac{(mn)!}{m!(n!)^m}\gamma_{mn}(a)$ for all $a\in A$
\end{enumerate}
A morphism between divided power algebras is an algebra morphism compatible with the operations $\gamma_n$.
\end{defi}
Recall that the operad $Com$ is an operad defined, in terms of $\∑$-modules, as:
\begin{align*}
	Com(n):=\left\{
	\begin{array}{lc}
		0,&\mbox{ if }n=0,\\
		\F\mbox{ (trivial representation)},&\mbox{ if }n\neq 0.
	\end{array}
	\right.
\end{align*}
One can apply \Cref{theoinv} on it. The category of $\Gamma(Com)$-algebras is isomorphic to the category of $\F$-vector spaces $A$ equipped with operations $\beta_{\lambda X_n,\r}$ for all $n\in\N$ and $\lambda\in\F$. Using the relations of \Cref{defbeta'}, it can be shown that the family of operations $(\gamma_n)_{n\in\N}:=(\beta_{X_n,(n)})_{n\in\N}$ endows $A$ with a structure of divided power algebra, with multiplication the operation $\beta_{X_2,(1,1)}$. Conversely, on any divided power algebra $(A,\gamma)$ one can create a family of operations $\beta_{*,*}$ by setting, for all $\r\in\Comp_p(n)$ all $(a_i)_{i\in[p]}$, and $\lambda\in\F$, 
\begin{equation}\label{eqCOM}
  \beta_{\lambda X_n,\r}(a_1,\dots,a_p)=\lambda\prod_{i\in[p]}\gamma_{r_i}(a_i).
\end{equation}
For other modern surveys on divided power algebras, see \cite{NS}, Section 3, \cite{BR}, Section 2 and \cite{BF}, Section 1.2.2.

\section{From invariants to coinvariants}\label{seccoinv}
The aim of this section is to improve the characterisation of $\Gamma(\P)$-algebras given by \Cref{theoinv} in the case where $\P$ is a reduced operad spanned by a set operad.

Formally, the goal of this section is to explain the definition of comparison functors that fit in a diagram of the following form:
$$\diag{\beta(\M)\ar@/_/[d]_{P_\M}&\beta'(\P)\ar@{_{(}->}[l]\ar@/_/[d]_{P_\P}\\
  \gamma(\M)\ar@/_/[u]_{O_\M}&\gamma'(\P)\ar@{_{(}->}[l]\ar@/_/[u]_{O_\P}
},$$
where the vertical functors will be shown to be isomorphisms under the right assumptions. Here, $\beta(\M)$ and $\beta'(\P)$ are the categories of generalised divided power modules and algebras defined in \Cref{subsecF}. The categories $\gamma(\M)$ and $\gamma'(\P)$ will be similar to the categories $\beta(\M)$ and $\beta'(\P)$ in that they have objects vector spaces endowed with polynomial operations subjected to a list of relations, except that the family of operations considered and the list of relations will be somewhat smaller and easier to manipulate.

In \Cref{seccoinv1}, we explain the definition of the categories $\gamma(\M)$ and $\gamma'(\P)$ in full details, we give the definition of the functors $P_\M$ and $P_\P$, and we show that, when $\M$ is endowed with a basis stable under the action of the symmetric groups, and when $\P$ is spanned by a set operad, these functors are isomorphisms of categories. This is the result of \Cref{theocoinv}. Together with \Cref{theoinv}, this result induces \Cref{theocoinvintro} as stated in the introduction of this article.

In \Cref{C}, we study the compatibility of the isomorphisms $\O_G$ of \Cref{B}, depending on a finite group $G$, with the properties of restriction to and induction from a subgroup of $G$. We construct morphisms that are used to prove \Cref{theocoinv}.

In \Cref{secsum}, we give a very brief summary of the results obtained in \Cref{secinv} and \Cref{seccoinv}.



\subsection{Construction of the functors \texorpdfstring{$P_\M$}{} and \texorpdfstring{$P_\P$}{}}\label{seccoinv1}
In this subsection, we consider the particular case of a $\∑$-module $\M$ endowed with a basis $\B$ stable by permutation, and the case of a reduced operad $\P$ spanned by a set operad $\B$ ( that is, $\P=\F[\B]$ as in \Cref{foncop}). 
We define categories $\gamma(\M)$ and $\gamma'(\P)$, which are similar to the categories $\beta(\M)$ and $\beta'(\P)$ of \Cref{secinv} in that they have objects vector spaces endowed with polynomial operations $\gamma_{[x]_{\r},\r}$ satisfying some relations, except that these operations are now indexed by $\r\in\Comp_p(n)$ and $[x]_\r\in\M(n)_{\∑_\r}$. We then construct a pair of functors $P_\M:\beta(\M)\to\gamma(\M)$ and $P_\P:\beta'(\P)\to\gamma'(\P)$. The main result of this section is \Cref{theocoinv}, which states that these functors are isomorphisms of categories with inverse functors $O_\M$ and $O_\P$. The proof of \Cref{theocoinv} uses results from \Cref{C}. Composed with the isomorphisms $F_\M$ and $F_\P$ of \Cref{secinv}, these functors give a refined characterisation of $\Gamma(\M)$-modules and $\Gamma(\P)$-algebras as generalised divided power algebras. This is \Cref{theocoinvintro} as it is stated in the introduction of this article.

In this section, we will use the definitions and notation defined in \Cref{Partitions}, \Cref{B} and in \Cref{subsecF}.

\begin{defi}\label{defgamma}
  Let $\M$ be a $\∑$-module such that, for all $n\in \N$, $\M(n)$ is endowed with a basis $\B(n)$ stable under the action of $\∑_n$. The category $\gamma(\M)$ has objects $(A,\gamma)$, where $A$ is an $\F$-vector space, and $\gamma_{*,*}$ is a family of operations $\gamma_{[x]_\r,\r}:A^{\times p}\to A$, where $\r\in\Comp_p(n)$ and $[x]_\r\in\∑_\r\backslash\B(n)$, and which satisfy:
    \begin{enumerate}[label=($\gamma$\arabic*),itemsep=5pt]
    \item\label{rel'perm}$\gamma_{[x]_{\r},\r}(a_1,\dots,a_p)=\gamma_{[\rho^*\cdot x]_{\r^\rho}, \r^\rho}(a_{\rho^{-1}(1)},\dots,a_{\rho^{-1}(p)})$ for all $\rho\in\∑_p$,  for all $\rho\in\∑_p$, where $\rho^*$ stands for the block permutation of blocks of size $(r_i)$ associated to $\rho$.

    \item\label{rel'0} $\gamma_{[x]_{(0,r_1,r_2,\dots,r_p)},(0,r_1,r_2,\dots,r_p)}(a_0,a_1,\dots,a_p)=\gamma_{[x]_{(r_1,\dots,r_p)},(r_1,r_2,\dots,r_p)}(a_1,\dots,a_p)$.

    \item\label{rel'lambda} $\gamma_{[x]_{\r},\r}(\lambda a_1,a_2,\dots,a_p)=\lambda^{r_1}\gamma_{[x]_{\r},\r}(a_1,\dots,a_p)\quad\forall \lambda\in \F$.

    \item\label{rel'repet} Let $\r\in \Comp_p(n)$, $\q\in \Comp_s(p)$ and $x\in \B(n)$. Then:
    $$\gamma_{[x]_\r,\r}(\undercount{q_1}{a_1,\dots,a_1},\undercount{q_2}{a_2,\dots,a_2},\dots,\undercount{q_s}{a_s,\dots,a_s})=\big(\frac{\val{Stab_{{\q\rhd\r}}(x)}}{\val{Stab_{\r}(x)}}\big)\gamma_{[x]_{\q\rhd\r},\q\rhd\r}(a_1,a_2,\dots,a_s).$$

    \item\label{rel'somme} 
    $$\gamma_{[x]_{\r,\r}}(a_0+a_1,a_2,\dots,a_p)=\sum_{l+m=r_1}\sum_{[y]_{\r\circ_1(l,m)}\in \∑_{\r\circ_1(l,m)}\backslash \Omega_{\∑_\r}(x)}\gamma_{[y]_{\r\circ_1(l,m)},\r\circ_1(l,m)}(a_0,a_1,\dots,a_p),$$
  \end{enumerate}
  and morphisms are the linear maps which are compatible with the operations $\gamma_{[x]_\r,\r}$.
\end{defi}
\begin{note*}
  Objects of $\gamma(\M)$ will be denoted as couples $(A,\gamma)$, where $A$ is an $\F$-vector space and $\gamma$ is a family of operations $\gamma_{[x]_\r,\r}:A^{\times p}\to A$ with $\r\in\Comp_p(n)$ and $[x]_\r\in\∑_\r\backslash\B(n)$.
\end{note*}
\begin{rema}
  This definition depends on the chosen basis $\B$ of $\M$. However, there is an isomorphism of categories between two categories $\gamma(\M)$ coming from two different bases of the same $\M$ stable under the actions of $\∑_n$. The proof of this fact is in the proof of \Cref{theocoinv}.
\end{rema}
\begin{rema}\label{rem'part}
  Those operations induce, for all, $\r\in\Comp_p(n)$, $R\in\Pi(\r;n)$ and all $[x]_R\in\∑_R\backslash\B(n)$, an operation $\gamma_{[x]_R,R}:A^{\times p}\to A$. Indeed, according to \Cref{transn}, there exists $\tau\in\∑_n$ such that $\tau(R)=\r$. We set:
$$\gamma_{[x]_R,R}(a_1,\dots,a_p):=\gamma_{[\tau\cdot x]_\r,\r}(a_1,\dots,a_p).$$
This does not depend on the choice of $\tau$.
\end{rema}
\begin{defi}\label{defgamma'}
  Let $\B$ be a set operad and $\P=\F[\B]$ be the operad spanned by $\B$, as described in \ref{foncop}. Suppose that $\P$ is reduced. Let $\gamma'(\P)$ be the full subcategory of $\gamma(\P)$ of objects $(A,\gamma)$ satisfying:
  \begin{enumerate}[label=($\gamma$6\alph*),itemsep=5pt]
    \item\label{rel'unit} $\gamma_{id_\P,(1)}(a)=a$, for all $a\in A$.

    \item\label{rel'comp} Let $\r\in \Comp_p(n)$, $x\in\B(n)$ and, for all $i\in[p]$, let $\q_i\in \Comp_{k_i}(m_i)$, $x_i\in\B(m_i)$ and $a_i=(a_{ij})_{j\in[k_i]}\in A^{\times k_i}$. Set $a=(a_{i,j})_{i\in[p],j\in[k_i]}$.
    Then:
    \begin{multline*}
      \gamma_{[x]_\r,\r}(\gamma_{[x_1]_{\q_1},\q_1}(a_{1}),\dots,\gamma_{[x_p]_{\q_p},\q_p}(a_{p}))=\\
      \Bigg(\frac{\val{Stab_{\r\diamond(\q_i)_{i\in[p]}}\bigg(\mu\Big(x\otimes \big(\bigotimes_{i=1}^px_i^{\otimes r_i}\big)\Big)\bigg)}}{\val{Stab_{\r}(x)}\prod_{i=1}^p\val{Stab_{\q_i}(x_i)}^{r_i}}\Bigg)\gamma_{\big[\mu\big(x\otimes \big(\bigotimes_{i=1}^px_i^{\otimes r_i}\big)\big)\big]_{r\diamond(\q_i)_{i\in[p]}},\r\diamond(q_i)_{i\in[p]}}(a),
    \end{multline*}
  \end{enumerate}
  where $\gamma_{\big[\mu\left(x\otimes \left(\bigotimes_{i=1}^px_i^{\otimes r_i}\right)\right)\big]_{r\diamond(\q_i)_{i\in[p]}},\r\diamond(q_i)_{i\in[p]}}$ is defined in \Cref{rem'part}.
\end{defi}
 \begin{rema}\label{rem'}
   Relation \ref{rel'repet} is equivalent to the following relation, which is \textit{a priori} weaker:
   \begin{multline*}
   	\gamma_{[x]_{(r_0,r_1,\dots,r_p)},\r}(a_1,a_1,a_2,\dots,a_p)=\\
   	\frac{\val{Stab_{(r_0+r_1,r_2,r_3,\dots,r_p)}(x)}}{\val{Stab_{(r_0,r_1,\dots,r_p)}(x)}}\gamma_{[x]_{(r_0+r_1,r_2,r_3,\dots,r_p)},(r_0+r_1,r_2,r_3,\dots,r_p)}(a_1,\dots,a_p).
   \end{multline*}
   Indeed, Relation \ref{rel'repet} is deduced from this one by iterating it, and using relation \ref{rel'perm}.
 \end{rema}
 \begin{lemm}\label{lem0}
  Let $\M$ be a $\∑$-module, endowed, for all $n\in\N$, with a basis $\B(n)$ stable under the action of $\∑_n$. Then, the operations $\beta_{\O_{\∑_\r}([y]_\r),\r}$, where $\O_{\∑_\r}$ is the morphism of \Cref{B} and $[y]_\r\in\∑_\r\backslash \B(n)$, generates all the operations $\beta_{*,*}$.
\end{lemm}
\begin{proof}
Relation \ref{rellin} implies that any basis of $\M(n)^{\∑_\r}$ furnishes generating operations. The Lemma of \Cref{B} shows that $\∑_\r\backslash\B(n)$ is a basis for $\M(n)_{\∑_\r}$, and that $\O_{\∑_\r}:\M(n)_{\∑_\r}\to\M(n)^{\∑_\r}$ is an isomorphism.
\end{proof}
 \begin{prop}\label{func1}
   Let $(A,\beta)$ be an object of $\beta(\M)$. Let us set, for all $\r\in\Comp_p(n)$ and all $[x]_\r\in\∑_\r\backslash\B(n)$:
  \begin{equation}\label{eqgamma}
    \gamma_{[x]_\r,\r}=\beta_{\O_\r([x]_\r),\r},
  \end{equation}
    where $\O_{\∑_\r}:\M(n)_{\∑_\r}\to\M(n)^{\∑_\r}$ is defined in \Cref{B}. This gives a functor 
    \begin{eqnarray*}
      P_\M:&\beta(\M)&\to\gamma(\M)\\
      &(A,\beta)&\mapsto(A,\gamma)
    \end{eqnarray*}
 \end{prop}
 \begin{prop}\label{propcoinv}
   Let $\M$ be a $\∑$-module such that, for all $n\in \N$, $\M(n)$ is endowed with a basis $\B(n)$ that is stable under the action of $\∑_n$. The functor $P_\M$ is an isomorphism of categories.
 \end{prop}
 \begin{theo}\label{theocoinv}
  If $\P$ is a reduced operad spanned by a set operad $\B$ (i.e. $\P=\F[\B]$), then the functor $P_\P$ restricts to a functor $P_\P:\beta'(\P)\to\gamma'(\P)$ and the latter is an isomorphism of categories.
 \end{theo}
 \begin{proof}[Proof of \Cref{func1}, \Cref{propcoinv} and \Cref{theocoinv}]\item 
  For this proof, we will need \Cref{g_1}, \Cref{g_2}, and \Cref{INDMULT}. Let $(A,\beta)$ be an object of $\beta(\M)$. The operations $\beta$ satisfy the relations \ref{relperm}, \ref{rel0} and \ref{rellambda} of \Cref{defbeta}, which, through \cref{eqgamma}, translate into \ref{rel'perm}, \ref{rel'0} and \ref{rel'lambda}. To translate Relation \ref{relrepet}, observe that
  \begin{align*}
    \gamma_{[x]_\r,\r}(\undercount{q_1}{a_1,\dots,a_1},\undercount{q_2}{a_2,\dots,a_2},\dots,\undercount{q_s}{a_s,\dots,a_s})&=\beta_{\O_\r\big([x]_\r\big),\r}(\undercount{q_1}{a_1,\dots,a_1},\undercount{q_2}{a_2,\dots,a_2},\dots,\undercount{q_s}{a_s,\dots,a_s})\\
    &=\beta_{\Big(\sum_{\sigma\in \∑_{\q\rhd\r}/\∑_\r}\sigma\cdot \O_\r\big([x]_\r\big)\Big),\ \q\rhd\r}(a_1,a_2,\dots,a_s)\\
  \end{align*}
  \Cref{g_1} then implies
  \begin{multline*}
     \gamma_{[x]_\r,\r}(\undercount{q_1}{a_1,\dots,a_1},\undercount{q_2}{a_2,\dots,a_2},\dots,\undercount{q_s}{a_s,\dots,a_s})=\\
     \beta_{\O_{\∑_{\q\rhd\r}}(Ind_\r^{\q\rhd\r}([x]_\r)),\ \q\rhd\r}(a_1,a_2,\dots,a_s)\\
     =\big(\frac{\val{Stab_{{\q\rhd\r}}(x)}}{\val{Stab_{\r}(x)}}\big)\beta_{\O_{\∑_{\q\rhd\r}}(\O_{\∑_{\q\rhd\r}}([x]_{\q\rhd\r})),\ \q\rhd\r}(a_1,a_2,\dots,a_s)\\
     =\big(\frac{\val{Stab_{{\q\rhd\r}}(x)}}{\val{Stab_{\r}(x)}}\big)\gamma_{[x]_{\q\rhd\r},\ \q\rhd\r}(a_1,a_2,\dots,a_s)
  \end{multline*}
  To translate Relation \ref{relsomme}, observe that
  \begin{align*}
    \gamma_{[x]_{\r,\r}}(a_0+a_1,a_2,\dots,a_p)&=\beta_{\O_{\∑_\r}([x]_{\r},\r)}(a_0+a_1,a_2,\dots,a_p)\\
    &=\sum_{l+m=r_1}\beta_{\O_{\∑_\r}([x]_{\r}),\r\circ_1(l,m)}(a_0,a_1,\dots,a_p)
  \end{align*}
  \Cref{g_2} then implies
  \begin{multline*}
    \gamma_{[x]_{\r,\r}}(a_0+a_1,a_2,\dots,a_p)=
    \sum_{l+m=r_1}\beta_{\O_{\∑_{\r\circ_1(l,m)}(Res_{\r\circ_1(l,m)}^{\r}([x]_\r)),\r\circ_1(l,m)}}(a_0,a_1,\dots,a_p)\\
    =\sum_{l+m=r_1}\sum_{[y]_{\r\circ_1(l,m)}\in\∑_{\r\circ_1(l,m)}\backslash\Omega_{\∑_r}(x)}\beta_{\O_{\∑_{\r\circ_1(l,m)}([y]_{\r\circ_1(l,m)})),\r\circ_1(l,m)}}(a_0,a_1,\dots,a_p)\\
    =\sum_{l+m=r_1}\sum_{[y]_{\r\circ_1(l,m)}\in\∑_{\r\circ_1(l,m)}\backslash\Omega_{\∑_r}(x)}\gamma_{[y]_{\r\circ_1(l,m)},\r\circ_1(l,m)}(a_0,a_1,\dots,a_p).
  \end{multline*}

   Assuming that $\M=\P=\F[\B]$ and that $(A,\beta)$ is an object of $\beta'(\P)$, then Relation \ref{relunit} of \Cref{defbeta'} translates into \ref{rel'unit} and, to translate Relation \ref{relcomp}, observe that
   \begin{align*}
     \gamma_{[x]_\r,\r}(\gamma_{[x_1]_{\q_1},\q_1}(a_{1}),\dots,\gamma_{[x_p]_{\q_p},\q_p}(a_{p}))&=\beta_{\O_{\∑_\r}([x]_\r),\r}(\beta_{\O_{\∑_{\q_1}}([x_1]_{\q_1}),\q_1}(a_{1}),\dots,\beta_{\O_{\∑_{\q_p}}([x_p]_{\q_p}),\q_p}(a_{p}))\\
     &=\beta_{\sum_{\tau}\tau\cdot\mu\left(\O_{\∑_\r}\left([x]_\r\right)\otimes \left(\bigotimes_{i=1}^p\O_{\∑_{\q_i}}([x_i]_{\q_i})\right)\right),\r\diamond(\q_i)_{i\in[p]}}(a),
   \end{align*}
   where $\tau$ ranges over $\∑_{\r\diamond(q_i)_{i\in[p]}}/(\prod_{i=1}^{p}\∑_{r_i}\wr \∑_{\q_i})$. \Cref{INDMULT} implies
   \begin{equation*}
     \gamma_{[x]_\r,\r}(\gamma_{[x_1]_{\q_1},\q_1}(a_{1}),\dots,\gamma_{[x_p]_{\q_p},\q_p}(a_{p}))=
     \beta_{\sum_{\tau}\tau\cdot\O_{\prod_{i=1}^p\∑_{r_i}\wr\∑_{\q_i}}(\mu'([x]_\r;[x_i]_{\q_i})),\r\diamond(\q_i)_{i\in[p]}}(a).
   \end{equation*}
   \Cref{g_1} implies
       \begin{multline*}
      \gamma_{[x]_\r,\r}(\gamma_{[x_1]_{\q_1},\q_1}(a_{1}),\dots,\gamma_{[x_p]_{\q_p},\q_p}(a_{p}))=\\
      \Bigg(\frac{\val{Stab_{\r\diamond(\q_i)_{i\in[p]}}\bigg(\mu\Big(x\otimes \big(\bigotimes_{i=1}^px_i^{\otimes r_i}\big)\Big)\bigg)}}{\val{Stab_{\r}(x)}\prod_{i=1}^p\val{Stab_{\q_i}(x_i)}^{r_i}}\Bigg)\gamma_{\left[\mu\left(x\otimes \left(\bigotimes_{i=1}^px_i^{\otimes r_i}\right)\right)\right]_{r\diamond(\q_i)_{i\in[p]}},\r\diamond(q_i)_{i\in[p]}}(a).
    \end{multline*}
   Thus, $P_{\M}(A,\beta)$ is an object of $\gamma(\M)$ and $P_{\P}(A,\beta)$ is an object of $\gamma'(\P)$. It is natural in $(A,\beta)$ if $\M$ and its basis $\B$ are fixed. 

   Let us construct an inverse functor. According to \Cref{lem0}, it is sufficient to define the operations $\beta_{\O_{\∑_\r}([x]_{\r}),\r}$ for $x$ ranging over  $\∑_\r\backslash\B(n)$. Then, setting:
$$\beta_{\O_{\∑_\r}(x),\r}:=\gamma_{[x]_\r,\r},$$
    one obtains a functor $O_{\M}:\gamma(\M)\to\beta(\M)$ which restricts to a functor $O_{\P}:\gamma'(\P)\to\beta'(\P)$. Indeed, we already checked that Relations \ref{rel'perm} to \ref{rel'comp} are equivalent to Relations \ref{relperm} to \ref{relcomp}. Moreover, those functors are clearly inverse to $P_{\M,\B}$ and $P_{\P,\B}$ respectively. Furthermore, since $\beta(\M)$ and $\beta'(\P)$ do not depend on the choice of $\B$, then changing the chosen basis would not change the isomorphism class of the category $\beta(\M)$, which justifies \textit{a posteriori} our choice to name the categories without referring to the chosen basis.
 \end{proof}
 \subsection{Technical results on the map \texorpdfstring{$\O$}{O}}\label{C}
In this subsection, we investigate the compatibility of the isomorphism $\O_G$ defined in \Cref{B}, and which depends on a finite group $G$, with the functorial properties for group representations of induction from and restriction to a subgroup $H$ of $G$. We define morphisms that we suggestively call $Ind$ and $Res$, with the appropriate groups as index, that are used in the proof of \Cref{theocoinv}.

Let $G$ be a finite group acting on a set $X$, let $H$ be a subgroup of $G$. Recall that there is a unique ring morphism $\Z\to\F$, that will be denoted $n\mapsto n_\F$. Recall also that \Cref{B} gives a linear isomorphism $\F[G\backslash X]\cong\F[X]_G$, so that $G\backslash X$ represents a basis of $\F[X]_G$.
\begin{prop}\label{g_1}
	Consider the morphism $Ind_H^G:\F[X]_H\to \F[X]_G$ induced by $Ind_H^G([x]_H)=\Big(\frac{ \val{Stab_G(x)}}{\val{Stab_H(x)}}\Big)_{\F}[x]_G$.
	The following diagram is commutative:
	\begin{align*}
		\diag{\F[X]_H\ar[r]^{\O_H}\ar[d]_{Ind_H^G}&\F[X]^H\ar[d]^{\sum_{g\in G/H}g\cdot id}\\
		\F[X]_G\ar[r]_{\O_G}&\F[X]^G}
	\end{align*}
\end{prop}
\begin{proof}
	Given two elements $x,y$ of $X$, if $x$ and $y$ are in the same class under the action of $H$, then they are in the same class under the action of $G$. Moreover, $Stab_H(x)$ and $Stab_H(y)$ are conjugated, and so are $Stab_G(x)$ and $Stab_G(y)$. Hence, the map $\F[X]\to \F[X]_G$ sending $x$ to $\Big(\frac{ \val{Stab_G(x)}}{\val{Stab_H(x)}}\Big)_{\F}[x]_G$ passes to the quotient by $H$. For $[x]_H\in H\backslash X$, one has
	\begin{align*}
		\sum_{g\in G/H}g\cdot \O_H([x]_H)&=\sum_{g\in G/H}g\cdot \sum_{h\in H/Stab_H(x)} h\cdot x\\
		&=\sum_{g\in G/Stab_H(x)}g\cdot x\\
		&=\sum_{g\in G/Stab_G(x)}g\cdot \sum_{h\in Stab_G(x)/Stab_H(x)} h\cdot x\\
		&=\O_G\bigg(\Big(\frac{ \val{Stab_G(x)}}{\val{Stab_H(x)}}\Big)_{\F}[x]_G\bigg).\qedhere
	\end{align*}
\end{proof}
\begin{prop}\label{g_2}
Let $x\in X$. The orbit of $x$ under the action of $G$, denoted $\Omega_G(x)$, is equipped with an action of $H$ induced by that of $G$. There is an injection $H\backslash \Omega_G(x)\inj H\backslash X$. Consider the morphism $Res_H^G:\F[X]_G\to \F[X]_H$ induced by $Res_H^G([x]_G)=\sum_{y\in H\backslash \Omega_G(x)}y$. The following diagram is commutative:
	\begin{align*}
		\diag{\F[X]_G\ar[r]^{\O_G}\ar[d]_{Res_H^G}&\F[X]^G\ar@{^(->}[d]\\
		\F[X]_H\ar[r]_{\O_H}&\F[X]^H}.
	\end{align*}
\end{prop}
\begin{proof}
	\begin{equation*}
		\O_H\bigg(\sum_{y\in H\backslash\Omega_G(x)}y\bigg)=\sum_{y\in H\backslash\Omega_G(x)}\O_H(y)
		=\sum_{y\in H\backslash\Omega_G(x)}\sum_{z\in \Omega_H(y)}z
		=\sum_{z\in \Omega_G(x)}z
		=\O_G([x]_G).\qedhere
	\end{equation*}
\end{proof}
\begin{prop}\label{INDMULT}
	Let $\P$ be an operad spanned by a set operad $\B$. Let $\r\in \Comp_p(n)$ and, for all $i\in[p]$, let $n_i\in\N$, and $H_i$ a subgroup of $\∑_{n_i}$.
	Let $x\in\B(n)$ and, for all $i\in[p]$, let $x_i\in\B(n_i)$. The  morphism $\mu':\P(n)_{\∑_\r}\otimes\Big(\bigotimes_i(\P(n_i)_{H_i}^{\otimes r_i})\Big)\to \P\Big(\sum_ir_in_i\Big)_{\prod_{i=1}^p\∑_{r_i}\wr H_i}$ induced by
	$$\mu'(x;(x_i)_i):=\Bigg(\frac{\val{Stab_{\prod_{i=1}^p\∑_{r_i}\wr H_i}\bigg(\mu\Big(x\otimes\Big(\bigotimes_{i}x_{i}^{\otimes r_i}\Big)\Big)\bigg)}}{\val{Stab_{\∑_\r}(x)}\prod_{i}\val{Stab_{H_i}(x_{i})}^{r_i}}\Bigg)_\F\mu\bigg(x\otimes\Big(\bigotimes_ix_i^{\otimes r_i}\Big)\bigg),$$
	
	makes the following diagram commute:
	$$\diag{
	\P(n)_{\∑_\r}\otimes\Big(\bigotimes_i(\P(n_i)_{H_i}^{\otimes r_i})\Big)\ar[rr]^{\O_{\∑_\r}\otimes\big(\bigotimes_i\O_{H_i}^{\otimes r_i}\big)}\ar[d]_{\mu'} &&\P(n)^{\∑_\r}\otimes\Big(\bigotimes_i\big((\P(n_i)^{H_i})^{\otimes r_i}\big)\Big)\ar[d]^{\mu}\\
	\P\Big(\sum_ir_in_i\Big)_{\prod_{i=1}^p\∑_{r_i}\wr H_i}\ar[rr]_{\O_{\prod_{i=1}^p\∑_{r_i}\wr H_i}} &&\P\Big(\sum_ir_in_i\Big)^{\prod_{i=1}^p\∑_{r_i}\wr H_i}
	}$$
\end{prop}
\begin{proof}
	Set $H:=Stab_{\prod_{i=1}^p\∑_{r_i}\wr H_i}\bigg(\mu\Big(x\otimes\big(\bigotimes_ix_i^{\otimes r_i}\big)\Big)\bigg)$. One has
	\begin{align*}
		\hspace{-32pt}\mu\bigg(\O_{\∑_\r}(x)\otimes\Big(\bigotimes_i\O_{H_i}(x_i)^{\otimes r_i}\Big)\bigg)&=\mu\bigg(\Big(\sum_{g\in\∑_\r/Stab_{\∑_\r}(x)}g\cdot x\Big)\otimes\Big(\bigotimes_i\Big(\sum_{h_i\in H_i/Stab_{H_i}(x_i)}h_i\cdot x_i\Big)^{\otimes r_i}\Big)\bigg)\\
		&=\sum_{g\in\big(\prod_{i=1}^p\∑_{r_i}\wr H_i\big)/Stab_{\∑_\r}(x)\wr\prod_i Stab_{H_i}(x_i)}g\cdot\mu\Big(x\otimes\Big(\bigotimes_ix_i^{\otimes r_i}\Big)\Big)\\
		&=\sum_{g\in \big(\prod_{i=1}^p\∑_{r_i}\wr H_i\big)/H}g\cdot\sum_{h\in H/Stab_{\∑_\r}(x)\wr\prod_i Stab_{H_i}(x_i)}h\cdot \mu\Big(x\otimes\Big(\bigotimes_ix_i^{\otimes r_i}\Big)\Big)\\
    &=\sum_{g\in \big(\prod_{i=1}^p\∑_{r_i}\wr H_i\big)/H}g\cdot\mu'(x;(x_i)_i).
	\end{align*}
  Since $Stab_{\prod_{i=1}^p\∑_{r_i}\wr H_i}\Big(\mu\big(x\otimes\big(\bigotimes_ix_i^{\otimes r_i}\big)\big)\Big)=Stab_{\prod_{i=1}^p\∑_{r_i}\wr H_i}\big(\tilde\mu(x;(x_i)_i)\big)$, one gets
  \begin{equation*}
    \mu\Big(\O_{\∑_\r}(x)\otimes\big(\bigotimes_i\O_{H_i}(x_i)^{\otimes r_i}\big)\Big)=\O_{\prod_i\∑_{r_i}\wr H_i}\big(\tilde\mu(x;(x_i)_i)\big).\qedhere
  \end{equation*}
\end{proof}
\subsection{Summary}\label{secsum}
For $\P=\F[\B]$ an operad spanned by a set operad $\B$, \Cref{theoinv} and \Cref{theocoinv} give isomorphisms among the categories $\Gamma(\P)_{alg}$, $\beta'(\P)$ and $\gamma'(\P)$. For $A$ an $\F$-vector space, this translates to an equivalence between the data of a morphism $f:\Gamma(\P,A)\to A$ compatible with the monad structure of $\Gamma(\P)$, with the two equivalent data of either:
\begin{itemize}
  \item operations $\beta_{*,*}$ such that, for all $\r\in\Comp_p(n)$, $x\in\P(n)^{\∑_\r}$ and $(a_i)_{1\le i\le p}\in A^{\times p}$
  $$\beta_{x,\r}(a_1,\dots,a_p)=f\bigg(\sum_{\sigma\in\∑_n/\∑_\r}\sigma\cdot x\otimes\sigma\cdot\Big(\bigotimes_{i=1}^pa_i^{r_i}\Big)\bigg),$$
  or,
  \item operations $\gamma_{*,*}$ such that, for all $\r\in\Comp_p(n)$, $x\in\B(n)_{\∑_\r}$ and $(a_i)_{1\le i\le p}\in A^{\times p}$
  $$\gamma_{[x]_\r,\r}(a_1,\dots,a_p)=f\bigg(\sum_{\sigma\in\∑_n/\∑_\r}\sigma\cdot \O_{\∑_\r}(x)\otimes\sigma\cdot\Big(\bigotimes_{i=1}^pa_i^{r_i}\Big)\bigg).$$
\end{itemize}

\section{The Operad \texorpdfstring{$Lev$}{Lev}}\label{secLEV}
In this section we give a characterisation of $\Gamma(Lev)$-algebras.

In \Cref{secdefiL}, we recall the basic definitions and notation regarding level algebras and the operad governing this structure.

In \Cref{chaLev}, we define the category $Step$ of vector spaces $A$ endowed with polynomial operations $\phi_{h,\r}:A^{\times p}\to A$ indexed by the functions $h:[n]\to\N$ that are constant on each part of the partition $\r$ of $n$ and satisfying $\sum_{i\in[n]}\frac{1}{2^{h(i)}}=1$, subjected to a list of relations. \Cref{theoL} states that the categories $Step$ and $\Gamma(Lev)_{alg}$ are isomorphic.

In the \Cref{proofL}, we give the proof of \Cref{theoL} by constructing an explicit pair of functors $U:\gamma'(Lev)\to Step$ and $V:Step\to \gamma'(Lev)$, and proving that these functors are inverse to each other. The functor $U$ equips any $\Gamma(Lev)$-algebra with an explicit family of polynomial operations.

\subsection{Level algebras: definitions}\label{secdefiL}
In this Subsection, we recall the definition of a level algebra, as given by D. Chataur and M. Livernet in \cite{CL}, following the definition of a depth algebra given by D.M. Davis in \cite{DD}, and recall a construction of the operad $Lev$ of level algebras, due to M. Livernet. We will use definitions and notation from \Cref{Partitions}.

\begin{defi}[Chataur and Livernet, \cite{CL}]\label{deflev}
  A level algebra is an $\F$-vector space $A$ equipped with a bilinear operation $*$ which is commutative (but not necessarily associative), satisfying:
  $$\forall (a,b,c,d)\in A^{\times 4}\quad (a*b)*(c*d)=(a*c)*(b*d).$$
\end{defi}
  For $n$ a positive integer, let $\L(n)$ be the set of partitions $I=(I_i)_{i\ge 0}$ of $[n]$ satisfying \newline$\sum_{i\in\N}\frac{\val{I_i}}{2^i}=1$. The set $\L(n)$ is endowed with the action of $\∑_n$ on partitions (see \Cref{transn}). 
\begin{note}
  For $I\in \L(n)$, $J\in \L(m)$ and $i\in[n]$, there is a unique $k\in\N$ such that $i\in I_k$ and a unique increasing bijection:
  $$b:[n]\setminus\{i\}\to \{1,\dots, i-1\}\cup\{i+m,\dots,n+m-1\}.$$
  The partition $I\circ_i J$ of $[n+m-1]$ is defined, for $j\ge 0$, by:
  $$(I\circ_iJ)_j:=\left\{\begin{array}{ccc}
  b(I_j),&\mbox{if}& j< k,\\
  b(I_k\setminus\{i\})\cup J_0+i-1,&\mbox{if}&j=k,\\
  b(I_j)\cup J_{j-k}+i-1,&\mbox{if}&j>k.
\end{array}\right.$$
\end{note}
\begin{lemm}
  The partition $I\circ_iJ$ is in $\L(m+n-1)$.
\end{lemm}
\begin{note}
  Set $1_\L:=(\{1\},\emptyset,\emptyset,\dots)\in\L(1)$.
 \end{note}
\begin{rema}\label{remlev}
  The set $\L(n)$ is in bijection with the set $\L'(n)$ of maps $h:[n]\to\N$ satisfying $\sum_{i=1}^n\frac{1}{2^{h(n)}}=1$.
  The bijection $\L(n)\to\L'(n)$ associates $I$ to the map $I^f$, which sends $x\in[n]$ to the unique $j\in\N$ such that $x\in I_j$. Its inverse $\L'(n)\to \L(n)$ sends $h$ to the partition $h^P:=(h^{-1}(j))_{j\in \N}$ of $[n]$.
  From the action of $\∑_n$ on $\L(n)$ one deduces $\sigma\cdot h(x)=h(\sigma^{-1}(x))$ and, for $h\in\L'(n)$, $g\in\L'(m)$ and $i\in[n]$: 
  $$(h\circ_ig)(x)=\left\{\begin{array}{ccc}
  h(x),&\mbox{if}&x\in\{1,\dots, i-1\},\\
  h(i)+g(x-i+1),&\mbox{if}&x\in\{i,\dots, m+i-1\},\\
  h(x-m+1),&\mbox{if}&x\in \{i+m,\dots,n+m-1\}.
  \end{array}\right.$$
  The unit of this operation is the map $1_{\L'}:[1]\to\N$ that sends $1$ to $0$.
\end{rema}
\begin{prop}
  The family $(\L(n))_{n\in\N}$, along with the operations $\circ_i:\L(m)\times\L(n)\to\L(m+n-1)$ for $i\in[m]$ and with $1_\L$, forms a set operad denoted by $\L$.
\end{prop}
 For $I\in \L(n)$, denote by $o(I)$ the height of $I$, that is the smallest $j\in\N$ such that for all $i>j$, $I_i=\emptyset$. If $h\in\L'(n)$, then $o(h)$ is the maximum of $h$. Note that, here, because the partitions $I\in Lev(n)$ are indexed from 0, $I$ can be considered as an ordered partition of $[n]$ into $o(n)+1$ parts.
\begin{defi}
  $Lev$ is the operad spanned by $\L$, that is $Lev=\F[\L]$, where the operad $\F[\B]$ for a set operad $\B$ is defined in \Cref{foncop}.
\end{defi}
\begin{prop}
  The category of $Lev$-algebras is exactly the category of level algebras.
\end{prop}
\subsection{Characterisation of divided power \texorpdfstring{$Lev$}{Lev}-algebras}\label{chaLev}
In this subsection, we give our notation for certain representatives of coinvariant elements of the operad $Lev$, then we give a definition of a category $Step$ of vector spaces equipped with polynomial operations indexed by these coinvariant elements. The result of \Cref{theoL} gives a characterisation of $\Gamma(Lev)$-algebras. More precisely, the operad $Lev$ being spanned by a set operad, \Cref{theoinv} and \Cref{theocoinv} give an isomorphism between $\Gamma(Lev)_{alg}$ and $\gamma'(Lev)$, and \Cref{theoL} states that $\gamma'(Lev)$ is isomorphic to $Step$.

In this subsection, the definitions and notation of \Cref{Partitions} will be broadly used.

\begin{note}\label{notaL}\item
\begin{itemize}
  \item For all $I\in\L(n)$, recall that $\∑_I=\prod_{i\in\N}\∑_{I_i}$. Note that $\∑_I=Stab(I)$.
  \item For all $R\in \Pi(\r,n)$, $\C_R$ is the set of maps $h:[n]\to\N$ which are constant on the sets $R_i$ for all $i$ and satisfy $\sum_{i=1}^n\frac{1}{2^{h(i)}}=1$. There is an inclusion $\C_R\subset\L'(n)$. For all refinements $Q$ of $R$, there is an inclusion $\omega_Q^R:\C_R\inj\C_Q$ (sometimes denoted $\omega$, when partitions $R$ and $Q$ can be deduced from the context).
  For $R\in \Pi(\r,n)$ and $\rho\in \∑_p$, there is a bijection:
  \begin{eqnarray*}
    &\C_R&\to\C_{\rho\cdot R}\\
    &h&\mapsto \rho^*\cdot h,
  \end{eqnarray*}
  with $\rho^*\in\∑_n$ the block permutation deduced from $\rho$ whose blocks have size $r_1,\dots,r_p$.
  
  Recall that there is a map $\iota:\Comp_p(n)\inj\Pi_p(n)$ which identifies $\underline{r} = (r_1,\dots,r_p)$ with the partition $\iota(\r)$ such that $R_k = \{r_1+\dots+r_{k-1}+1,\dots,r_1+\dots+r_{k-1}+r_k\}$. We will abuse notation and use the expression $\C_{\r}$ for $\C_{\iota(\r)}$.
  \item Every $[h]_{\r}\in\∑_\r\backslash\L'(n)$ has a unique representative $h$ which is non-decreasing on each $\r_i$, and the composition $\r\wedge h\in \Comp_{(o(h)+1)p}(n)$ (see \ref{wedge}) satisfies $Stab_\r(h)=\∑_{\r\wedge h}$.
  \item 
  The next propositions will involve operations of the type $\phi_{h,\r}:A^{\times p}\to A$, indexed by $\r\in\Comp_p(n)$ and $h\in\C_\r$. These operations also induce operations $\phi_{h,R}:A^{\times p}\to A$ for all partitions $R\in\Pi(\r,n)$ and $h\in\C_R$. Indeed, following \Cref{transn}, there exists $\tau\in\∑_n$ such that $\tau(R)=\r$ and one defines:
  $$\phi_{h,R}(a_1,\dots,a_p):=\phi_{\tau\cdot h,\r}(a_1,\dots,a_p),$$
  which does not depend on the chosen $\tau$.
\end{itemize}
\end{note}
\begin{rema}
  The partial order on the set of compositions of $n$, given by $\r\le\q$ if and only if $\q$ is a refinement of $\r$, induces a partial order on the set $\{\C_\r\}$ where $\r$ ranges over the compositions of $n$. This partial order is the inclusion order. The minimal element is $\C_{(n)}$, which is empty, unless $n=2^k$ with $k\in\N$, in which case it only contains the constant function equal to $k$. The maximal element is $\C_{d_n}$ ($d_n$ being the discrete composition $(\underset{n}{\underbrace{1,\dots,1}})$), which is equal to $\L'(n)$.
\end{rema}
\begin{defi}\label{defL}
  A step algebra $(A,\phi)$ is an $\F$-vector space $A$ equipped, for all $\r\in \Comp_p(n)$ and $h\in\C_\r$, of an operation:
  $$\phi_{h,\r}:A^{\times p}\to A$$
  satisfying the following relations:
  \begin{enumerate}[label=(S\arabic*),itemsep=5pt]
    \item\label{relLperm} $\phi_{\rho^*\cdot h,\r^\rho}(a_{\rho^{-1}(1)},\dots,a_{\rho^{-1}(p)})=\phi_{h,\r}(a_1,\dots,a_p)$.
    \item\label{relL0} $\phi_{h,(0,r_1,\dots,r_p)}(a_0,a_1,\dots,a_p)=\phi_{h,(r_1,\dots,r_p)}(a_1,\dots,a_p)$.
    \item\label{relLlambda} $\phi_{h,\r}(\lambda a_1,a_2,\dots,a_p)=\lambda^{r_1}\phi_{h,\r}(a_1,\dots,a_p)$.
    \item\label{relLrepet}$\binom{r_1}{l}\phi_{h,\r}(a_1,\dots,a_p)=\phi_{\omega_{\r\circ_1(l,m)}^\r(h),\r\circ_1(l,m)}(a_1,a_1,a_2,\dots,a_p)$ for all $l,m\in\Comp_2(r_1)$.
    \item\label{relLsomme}$\phi_{h,\r}(a+b,a_2,\dots,a_p)=\sum_{l+m=r_1}\phi_{\omega_{\r\circ_1(l,m)}^{\r}(h),\r\circ_1(l,m)}(a,b,a_2,\dots,a_p)$.
    \item\label{relLunit}$\phi_{1_{\L'},(1)}(a)=a$.
    \item\label{relLcomp} Let $\r\in \Comp_p(n)$ and for all $i\in[p]$, let $\q_i\in \Comp_{k_i}(m_i)$. Then,
    \begin{multline*}
    	\phi_{h,\r}(\phi_{g_i,\q_i}(a_{ij})_{j\in[k_i]})_{i\in[p]}=\\
    	\bigg(\prod_{i\in[p]}\frac{1}{r_i!}\big(\prod_{j\in[k_i]}\frac{(r_iq_{ij})!}{(q_{ij}!)^{r_i}}\big)\bigg)\phi_{\mu\left(h\otimes\left(\bigotimes_{i\in[p]}g_i^{\otimes r_i}\right)\right),\r\diamond(\q_i)_{i\in[p]}}(a_{ij})_{i\in[p],j\in[k_i]},
    \end{multline*}
  \end{enumerate}
  where $\phi_{\mu\left(h\otimes\left(\bigotimes_{i\in[p]}g_i^{\otimes r_i}\right)\right),\r\diamond(\q_i)_{i\in[p]}}$ is defined in the last item of \ref{notaL}.

  A morphism of step algebras $f:A\to A'$ is a linear map $f:A\to A'$ satisfying, for all $\r\in \Comp_p(n)$, all $h\in\C_\r$ and all $a_1,\dots,a_p\in A$,
  $$f(\phi_{h,\r}(a_1,\dots,a_p))=\phi_{h,\r}(f(a_1),\dots,f(a_p)).$$
  Step algebras and their morphisms form a category denoted by $Step$.
\end{defi}
\begin{rema}\label{remLcomp}
  Relation \ref{relLcomp} makes sense, because $\mu\bigg(h\otimes\Big(\bigotimes_{i\in[p]}g_i^{\otimes r_i}\Big)\bigg)$ is constant on $\r\diamond(\q_i)_{i\in[p]}$.
\end{rema}
\begin{rema}
  Relations \ref{relL0} to \ref{relLsomme} are expressed here by acting on the first variable, but can be rewritten, thanks to Relation \ref{relLperm}, as acting on another variable (in a similar way as in \Cref{remacting}).
\end{rema}
\begin{theo}\label{theoL}
  There is an isomorphism of categories: 
  $$\Gamma(Lev)_{alg}\to Step.$$
\end{theo}
\begin{proof} 
Since $Lev=\F[\L]$, we can apply \Cref{theoinv} and \Cref{theocoinv} to obtain an isomorphism $\Gamma(Lev)_{alg}\to \gamma'(Lev)$. The next section proves in four steps that there is an isomorphism $\gamma'(Lev)\to Step$.
\end{proof}
\subsection{Proof of \texorpdfstring{\Cref{theoL}}{}}\label{proofL}
In this subsection, we develop the proof of \Cref{theoL}, by constructing an explicit pair of isomorphisms of categories $U:\gamma'(Lev)\leftrightarrow Step$ and $V:Step\to\gamma'(Lev)$. The definitions and notation introduced in Subsections \ref{chaLev} and \ref{Partitions} will be broadly used.
\subsubsection*{Step 1: Construction of a functor \texorpdfstring{$U:\gamma'(Lev)\to Step$}{F:gamma'(Lev)—>Step}}

  Let $(A,\gamma)$ be an object of $\gamma'(Lev)$. For $\r\in \Comp_p(n)$ and $h\in\C_\r$, set:
  \begin{equation}\label{eqphi}
    \varphi_{h,\r}:=\gamma_{[h^P]_\r,\r}
  \end{equation}
  Let us show that $(A,\varphi)$ is a step algebra.

  Relations \ref{relLperm}, \ref{relL0} and \ref{relLlambda} follow directly from relations \ref{rel'perm}, \ref{rel'0} and \ref{rel'lambda}. Let us show that relation \ref{relLrepet} is satisfied. By definition, one has
  \begin{equation*}
    \binom{r_1}{l}\varphi_{h,\r}(a_1,\dots,a_p)=\binom{r_1}{l}\gamma_{[h^P]_\r,\r}(a_1,\dots,a_p).
  \end{equation*}
  Since $h\in\C_\r$, then $\∑_{\r\circ_1(l,m)}\subseteq\∑_\r\subseteq Stab(h)$ and
  $$\frac{\val{Stab_{\r}(h^P)}}{\val{Stab_{\r\circ_1(l,m)}(h^P)}}=\frac{\val{\∑_\r}}{\val{\∑_{\r\circ_1(l,m)}}}=\binom{r_1}{l},$$
  thus, using relation \ref{rel'repet}:
  $$\binom{r_1}{l}\varphi_{h,\r}(a_1,\dots,a_p)=\gamma_{[h^P]_{\r\circ_1(l,m)},\r\circ_1(l,m)}(a_1,a_1,a_2,\dots,a_p),$$
  which is equal to $\varphi_{\omega_{\r\circ_1(l,m)}^\r(h),\r\circ_1(l,m)}(a_1,a_1,a_2,\dots,a_p)$.
  Let us show relation \ref{relLsomme}:
  \begin{align*}
    \varphi_{h,\r}(a+b,r_2,\dots,r_p)&=\gamma_{[h^P]_\r,\r}(a+b,r_2,\dots,r_p)\\
    &\overeg{Relation \ref{rel'somme}}\sum_{l+m=r_1}\sum_{I\in\∑_{\r\circ_1(l,m)}\backslash\Omega_{\∑_\r}(h^P)}\gamma_{[I]_{\r\circ_1(l,m)},\r\circ_1(l,m)}(a,b,a_2,\dots,a_p).
  \end{align*}
  One has $\Omega_{\∑_\r}(h^P)=\{h^P\}$ for $\∑_\r\subseteq Stab(h)$, thus:
  \begin{align*}
    \varphi_{h,\r}(a+b,r_2,\dots,r_p)&=\sum_{l+m=r_1}\gamma_{[h^P]_{\r\circ_1(l,m)},\r\circ_1(l,m)}(a,b,a_2,\dots,a_p)\\
    &=\sum_{l+m=r_1}\varphi_{\omega_{\r\circ_1(l,m)}^{\r}(h),\r\circ_1(l,m)}(a,b,a_2,\dots,a_p).
  \end{align*}
  Relation \ref{relLunit} follows directly from \ref{rel'unit}. Let us finally show relation \ref{relLcomp}:
  \begin{equation*}
    \varphi_{h,\r}(\varphi_{g_i,\q_i}(a_{ij})_{j\in[k_i]})_{i\in[p]}=\gamma_{[h^P]_\r,\r}(\gamma_{[g_i^P]_{\q_i},\q_i}(a_{ij})_{j\in[k_i]})_{i\in[p]},
  \end{equation*}
  which is equal, according to \ref{rel'comp}, to:
  $$\Bigg(\frac{\val{Stab_{\r\diamond(\q_i)_{i\in[p]}}\bigg(\mu\Big(h\otimes \big(\bigotimes_{i=1}^pg_i^{\otimes r_i}\big)\Big)\bigg)}}{\val{Stab_{\r}(h)}\prod_{i=1}^p\val{Stab_{\q_i}(g_i)^{r_i}}}\Bigg)\gamma_{\left[\mu\left(h\otimes \left(\bigotimes_{i=1}^pg_i^{\otimes r_i}\right)\right)^P\right]_{r\diamond(\q_i)_{i\in[p]}},\r\diamond(q_i)_{i\in[p]}}(a_{ij})_{i\in[p],j\in[k_i]}.$$
  Since $h\in \C_\r$ and, for all $i\in[p]$, $g_i\in\C_{\q_i}$, then $\∑_\r\subseteq Stab(h)$ and $\∑_{\q_i}\subseteq Stab(g_i)$ and, according to Remark \ref{remLcomp}, $\∑_{\r\diamond(\q_i)}\subseteq Stab\bigg(\mu\Big(h\otimes \big(\bigotimes_{i=1}^pg_i^{\otimes r_i}\big)\Big)\bigg)$. 
  It follows that
  \begin{equation*}
    \frac{\val{Stab_{\r\diamond(\q_i)_{i\in[p]}}\bigg(\mu\Big(h\otimes \big(\bigotimes_{i=1}^pg_i^{\otimes r_i}\big)\Big)\bigg)}}{\val{Stab_{\r}(h)}\prod_{i=1}^p\val{Stab_{\q_i}(g_i)^{r_i}}}=\frac{\val{\∑_{\r\diamond(\q_i)_{i\in[p]}}}}{\val{\∑_\r}\times\prod_{i\in[p]}\val{\∑_{\q_i}}^{r_i}}
    =\bigg(\prod_{i\in[p]}\frac{1}{r_i!}\big(\prod_{j\in[k_i]}\frac{(r_iq_{ij})!}{(q_{ij}!)^{r_i}}\big)\bigg),
  \end{equation*}
  and
  \begin{multline*}
    \gamma_{\left[\mu\left(h\otimes \left(\bigotimes_{i=1}^pg_i^{\otimes r_i}\right)\right)^P\right]_{r\diamond(\q_i)_{i\in[p]}},\r\diamond(q_i)_{i\in[p]}}(a_{ij})_{i\in[p],j\in[k_i]}=\\
    \varphi_{\mu\left(h\otimes\left(\bigotimes_{i\in[p]}g_i^{\otimes r_i}\right)\right),\r\diamond(\q_i)_{i\in[p]}}(a_{ij})_{i\in[p],j\in[k_i]}.
  \end{multline*}
  Thus $(A,\varphi)$ is a step algebra.  
  Moreover, if $(A,\gamma)$ and $(A',\gamma)$ are two objects in $\gamma'(Lev)$ and if $f:A\to A'$ is a morphism in $\gamma'(Lev)$, then, by definition, $f$ is compatible with the operations $\gamma_{[I]_\r,\r}$, and so, according to the definition of $\varphi_{h,\r}$, $f$ is a morphism of step algebras. As a consequence,
  \begin{eqnarray*}
    U:&\gamma'(Lev)&\to Step\\
    &(A,\gamma)&\mapsto(A,\varphi)
  \end{eqnarray*}
is a functor.
\subsubsection*{Step 2: Construction of a functor \texorpdfstring{$V:Step\to\gamma'(Lev)$}{V:Step—>Gamma(Lev)-algebras}}

  Let $(A,\phi)$ be a step algebra. Let us set, for $\r\in \Comp_p(n)$, $[I]_\r\in\∑_{\r}\backslash\L(n)$ and $(a_1,\dots,a_p)\in A^{\times p}$:
  $$\theta_{[I]_\r,\r}(a_1,\dots,a_p):=\phi_{I^f,\r\wedge I^f}(\underset{o(I)+1}{\underbrace{a_1,\dots,a_1}},\dots,\underset{o(I)+1}{\underbrace{a_p,\dots,a_p}}),$$
  where, $I^f$ is the representative of $[I^f]_{\r}$ which is non-decreasing on each $\r_i$. Let us show that the operations $\theta_{[I]_{\r},\r}$ satisfy relations \ref{rel'perm} to \ref{rel'somme} of \Cref{defgamma} and relations \ref{rel'unit} and \ref{rel'comp} of \Cref{defgamma'}, making $A$ into an object of $\gamma'(Lev)$.
  Let us show that relation \ref{rel'perm} is satisfied.
For $(a_1,\dots,a_m)\in A^{\times m}$ and $\sigma\in\∑_m$, set $\sigma\cdot(a_1,\dots,a_m)=(a_{\sigma^{-1}(1)},\dots,a_{\sigma^{-1}(m)})$.
Let $\rho\in\∑_p$. On one hand, $\rho$ induces a block permutation $\rho^*\in\∑_n$ whose blocks have size $r_1,\dots,r_p$. On the other hand, $\rho$ also induces a block permutation $\rho^\triangle\in\∑_{(o(I)+1)p}$ whose blocks all have size $o(I)+1$ and this permutation itself induces a block permutation $(\rho^\triangle)^*\in\∑_n$ whose blocks have size $(\r\wedge I^f)_{1,0},\dots,(\r\wedge I^f)_{p,o(I)}$. It is clear that $(\rho^\triangle)^*=\rho^*$ and that
$$\r^\rho\wedge(\rho^*\cdot I^f)=(\r\wedge I^f)^{\rho^\triangle}.$$

It then follows that:
\begin{align*}
  \theta_{[\rho^*\cdot I]_{\r^\rho},\r^\rho}(\rho\cdot(a_1,\dots,a_p))&=\phi_{\rho^*\cdot I^f,\r^\rho\wedge(\rho^*\cdot I^f)}(\rho^\triangle\cdot(\underset{o(I)+1}{\underbrace{a_1,\dots,a_1}},\dots,\underset{o(I)+1}{\underbrace{a_p,\dots,a_p}}))\\
  &=\phi_{(\rho^\triangle)^*\cdot I^f,\rho^\triangle\cdot(\r\wedge I^f)}(\rho^\triangle\cdot(\underset{o(I)+1}{\underbrace{a_1,\dots,a_1}},\dots,\underset{o(I)+1}{\underbrace{a_p,\dots,a_p}}))\\
  &\overeg{Relation \ref{relLperm}}\phi_{I^f,\r\wedge I^f}(\underset{o(I)+1}{\underbrace{a_1,\dots,a_1}},\dots,\underset{o(I)+1}{\underbrace{a_p,\dots,a_p}})\\
  &=\theta_{[I]_\r,\r}(a_1,\dots,a_p).
\end{align*}
Let us check \ref{rel'0}. Note that:
$$(0,r_1,\dots,r_p)\wedge I^f=(\underset{o(I)+1}{\underbrace{0,\dots,0}},(\r\wedge I^f)_{1,0},\dots,(\r\wedge I^f)_{p,o(I)}).$$
One has:
\begin{align*}
  \theta_{[I]_{(0,r_1,\dots,r_p)},(0,r_1,\dots,r_p)}(a_0,a_1,\dots,a_p)&=\phi_{I^f,(0,r_1,\dots,r_p)\wedge I^f}(\underset{o(I)+1}{\underbrace{a_0,\dots,a_0}},\dots,\underset{o(I)+1}{\underbrace{a_p,\dots,a_p}})\\
  &\overeg{Relation \ref{relL0}}\phi_{I^f,\r\wedge I^f}(\underset{o(I)+1}{\underbrace{a_1,\dots,a_1}},\dots,\underset{o(I)+1}{\underbrace{a_p,\dots,a_p}}).
\end{align*}
Let us show relation \ref{rel'lambda}. Let $\lambda\in\F$. Because $(I_i)_{0\le i\le o(I)}$ is a partition of $[n]$, $((\r\wedge I^f)_{1j})_{0\le j\le o(I)}$ is a composition of $\r_1$, and so $\sum_{j=0}^{o(I)}(\r\wedge I^f)_{1j}=r_1$. One then has:
\begin{align*}
  \theta_{[I]_\r,\r}(\lambda a_1,a_2,\dots,a_p)&=\phi_{I^f,\r\wedge I^f}(\underset{o(I)+1}{\underbrace{\lambda a_1,\dots,\lambda a_1}},\underset{o(I)+1}{\underbrace{a_2,\dots,a_2}},\dots,\underset{o(I)+1}{\underbrace{a_p,\dots,a_p}})\\
  &\overeg{Relation \ref{relLlambda}}\big(\prod_{j=0}^{o(I)}\lambda^{(\r\wedge I^f)_{1j}}\big)\phi_{I^f,\r\wedge I^f}(\underset{o(I)+1}{\underbrace{a_1,\dots,a_1}},\dots,\underset{o(I)+1}{\underbrace{a_p,\dots,a_p}})\\
  &=\lambda^{r_1}\theta_{[I]_\r,\r}(a_1,\dots,a_p).
\end{align*}
Let us show \ref{rel'repet}. Set $\r=(r_0,\dots,r_p)$ and $\r'=(r_0+r_1,r_2,\dots,r_p)$. Let $J\in\L(n)$ and $I^f$ be a representative of $[J^f]_{\r}$ non-decreasing on each $\r_i$. One obtains a composition $(\r\wedge I^f)_{0\le i\le p,0\le j\le o(J)}$. There exists a permutation $\rho\in\∑_{(o(J)+1)(p+1)}$ which permutes the $2(o(J)+1)$ first parts of $\r\wedge I^f$ and stabilises the others, such that:
$$\rho\cdot(\r\wedge I^f)=((\r_0\cap I_0),(\r_1\cap I_0),(\r_0\cap I_1),(\r_1\cap I_1)\dots,(\r_0\cap I_{o(J)}),(\r_1\cap I_{o(J)}),\dots).$$
The permutation $\rho$ induces a block permutation $\rho^*\in\∑_n$ whose blocks have size $(\r\wedge I)_{00},\dots,(\r\wedge I)_{po(J)}$. One then has $\rho^*\in\∑_{\r'}$ and $\rho^*\cdot I^f$ is a representative of $[J^f]_{\r'}$ which is non-decreasing on the $\r'_i$.

One obtains another composition: $(\r'\wedge \rho^*\cdot I^f)_{1\le i\le p,0\le j\le o(J)}$. Note that:
\begin{equation*}
  \frac{\val{Stab_{\r'}(J)}}{\val{Stab_{\r}(J)}}=\frac{\prod_{j=0}^{o(J)}\val{(\r'_1)\cap J_j}!}{\prod_{j=0}^{o(J)}\val{\r_0\cap J_j}!\val{\r_1\cap J_j}!}
  =\prod_{j=0}^{o(J)}\binom{(\r'\wedge \rho^*\cdot I^f)_{1j}}{(\r\wedge I^f)_{0j}},
\end{equation*}
and that, for all $j\in\{0,\dots, o(J)\}$:
$$(\r'\wedge \rho^*\cdot I^f)_{1j}=(\r\wedge I^f)_{0j}+(\r\wedge I^f)_{1j}.$$
Thus:
\begin{align*}
 \hspace{-7.40134pt} \theta_{[J]_{\r},\r}(a_1,a_1,a_2,\dots,a_p)&=\phi_{I^f,\r\wedge I^f}(\underset{o(J)+1}{\underbrace{a_1,\dots,a_1}},\underset{o(J)+1}{\underbrace{a_1,\dots,a_1}},\dots,\underset{o(J)+1}{\underbrace{a_p,\dots,a_p}})\\
  &\overeg{Relation \ref{relLperm}}\phi_{\omega_{\rho\cdot(\r\wedge I^f)}^{\r'\wedge \rho^*\cdot I^f}(\rho^*\cdot I^f),\rho\cdot(\r\wedge I^f)}(\underset{2(o(J)+1)}{\underbrace{a_1,\dots,a_1}},\dots,\underset{o(J)+1}{\underbrace{a_p,\dots,a_p}})\\
  &\overeg{Relation \ref{relLrepet}}\prod_{j=0}^{o(J)}\binom{(\r'\wedge \rho^*\cdot I^f)_{1j}}{(\r\wedge I^f)_{0j}}\phi_{\rho^*\cdot I^f,\r'\wedge \rho^*\cdot I^f}(\underset{o(I)+1}{\underbrace{a_1,\dots,a_1}},\dots,\underset{o(I)+1}{\underbrace{a_p,\dots,a_p}})\\
  &=\frac{\val{Stab_{\r'}(J)}}{\val{Stab_{\r}(J)}}\theta_{[J]_{\r'},\r'}(a_1,\dots,a_p).
\end{align*}
According to \Cref{rem'}, this is enough to prove relation \ref{rel'repet}.

Let us check relation \ref{rel'somme}. One has:
\begin{equation*}
  \theta_{[I]_{\r,\r}}(a+b,a_2,\dots,a_p)=\phi_{I^f,\r\wedge I^f}(\underset{o(I)+1}{\underbrace{a+b,\dots,a+b}},\underset{o(I)+1}{\underbrace{a_2,\dots,a_2}},\dots,\underset{o(I)+1}{\underbrace{a_p,\dots,a_p}}),
\end{equation*}
which is equal, according to relation \ref{relLsomme}, to:

$$
	\sum\phi_{\omega_{(\r\wedge I^f)^{\nu}}^{(\r\wedge I^f)}(I^f),(\r\wedge I^f)^{\nu}}(\underset{o(I)+1}{\underbrace{a,b,\dots,a,b}},\underset{o(I)+1}{\underbrace{a_2,\dots,a_2}},\dots,\underset{o(I)+1}{\underbrace{a_p,\dots,a_p}}),
$$
where the sum ranges over the set of families $(\nu_{1j},\nu_{2,j})\in\prod_{j=0}^{o(I)}\Comp_2((\r\wedge I^f)_{1j})$ and where $(\r\wedge I^f)^{\nu}=((\dots((\r\wedge I^f)\circ_{io(h)}(\nu_{1o(h)},\nu_{2o(h)}))\circ_{1(o(h)-1)}(\nu_{1(o(h)-1)},\nu_{2(o(h)-1)}\dots)\circ_{10}(\nu_{10},\nu_{20}))$. Let $(\nu_{1j},\nu_{2j})_{0\le j\le o(I)}$ be such that $\nu_{1j}+\nu_{2j}=(\r\wedge I^f)_{1j}$ for all $j\in\{0,\dots,o(I)\}$. Set $l:=\sum_{j=0}^{o(I)}\nu_{1j}$ and $m:=\sum_{j=0}^{o(I)}\nu_{2j}$.
 The family of vectors $(a,b,\dots,a,b)$ can be sorted in order to put the occurrences of $a$ before the occurrences of $b$. Let $\rho\in\∑_{(o(I)+1)p}$ be the permutation stabilising the integers greater than $2(o(I)+1)$ and whose graph on the first $2(o(I)+1)$ integers is:
$$\big(\rho(1),\dots,\rho(2(o(I)+1))\big)=\big(1,3,5,\dots,2o(I)+1,2,4,6,\dots,2(o(I)+1)\big).$$
The permutation $\rho$ induces a block permutation $\rho^*\in\∑_n$ whose blocks have size $(\r\wedge I^f)^{\nu}_{ij}$, and it is an $(l,m)$-shuffle. One checks that $\rho^*\cdot I^f$ is the element of $Sh(l,m)I^f$ satisfying, for all $j\in\{0,\dots, o(I)\}$, 
$$(\nu_{1j})_{0\le j\le o(h)}=\val{[l]\cap (\rho^*\cdot I^f)^{-1}(\{j\})},$$
and
$$(\nu_{2j})_{0\le j\le o(h)}=\val{l+[m]\cap (\rho^*\cdot I^f)^{-1}(\{j\})},$$
see \ref{lemmOmega}, and that $\rho\cdot(\r\wedge I^f)^{\nu}=(\r\circ_1(l,m)\wedge \rho^*\cdot I^f)$. One then has:
\begin{multline*}
  \theta_{[I]_{\r},\r}(a+b,a_2,\dots,a_p)=\\
  \sum_{l+m=r_1}\sum_{J\in Sh(l,m)I}\phi_{J^f,(\r\circ_1(l,m)\wedge J^f)}(\underset{o(I)+1}{\underbrace{a,\dots,a}},\underset{o(I)+1}{\underbrace{b,\dots,b}},\underset{o(I)+1}{\underbrace{a_2,\dots,a_2}},\dots,\underset{o(I)+1}{\underbrace{a_p,\dots,a_p}}),
\end{multline*}
which is equal, according to \Cref{lemmOmega}, to:
$$\sum_{l+m=r_1}\sum_{[J]_{\r\circ_1(l,m)}\in\∑_{\r\circ_1(l,m)}\backslash\Omega_{\∑_\r}(I)}\theta_{[J]_{\r\circ_1(l,m)},\r\circ_1(l,m)}(a,b,a_2,\dots,a_p).$$
Relation \ref{rel'unit} is easy to check. Indeed, $1_\L^f=1_{\L'}$ $( (1)\wedge 1_{\L'})=(1)$, and so:
\begin{equation*}
  \theta_{[1_\L]_{(1)},(1)}(a)=\phi_{1_{\L'},(1)}(a)\overeg{Relation \ref{relLunit}}a\quad.
\end{equation*}
Let us now check Relation \ref{rel'comp}.
Let $\r\in \Comp_p(n)$ and let $I\in\L(n)$ be such that $I^f$ is non-decreasing on $\r_i$. For all $i\in[p]$, let $\q_i\in \Comp_{k_i}(m_i)$ and let $J_i\in\L(m_i)$ be such that $J_i^f$ is non-decreasing on each $\q_{ij}$ and let $(a_{ij})_{j\in[k_i]}\in A^{\times k_i}$. One has:
$$\hspace{-10cm}\theta_{[I]_{\r},\r}(\theta_{[J_i]_{\q_i},\q_i}(a_{ij})_{j\in[k_i]})_{i\in[p]}=$$
\begin{multline*}
  \phi_{I^f,(\r\wedge I^f)}\bigg(\\
  \underset{o(I)+1}{\underbrace{\phi_{J_1^f,(\q_1\wedge J_1^f)}(\underset{o(J_1)+1}{\underbrace{a_{11},\dots,a_{11}}},\dots,\underset{o(J_1)+1}{\underbrace{a_{1k_1},\dots,a_{1k_1}}}),\dots,\phi_{J_1^f,(\q_1\wedge J_1^f)}(\underset{o(J_1)+1}{\underbrace{a_{11},\dots,a_{11}}},\dots,\underset{o(J_1)+1}{\underbrace{a_{1k_1},\dots,a_{1k_1}}})}},\\
  \dots,\\
  \underset{o(I)+1}{\underbrace{\phi_{J_p^f,(\q_p\wedge J_p^f)}(\underset{o(J_p)+1}{\underbrace{a_{p1},\dots,a_{p1}}},\dots,\underset{o(J_p)+1}{\underbrace{a_{pk_p},\dots,a_{pk_p}}}),\dots,\phi_{J_p^f,(\q_p\wedge J_p^f)}(\underset{o(J_p)+1}{\underbrace{a_{p1},\dots,a_{p1}}},\dots,\underset{o(J_p)+1}{\underbrace{a_{pk_p},\dots,a_{pk_p}}}}})\bigg),
\end{multline*}
which is equal, according to relation \ref{relLcomp}, to:
\begin{multline*}
  \bigg(\prod_{i\in[p],j\in\{0,\dots,o(I)\}}\frac{1}{(\r\wedge I^f)_{ij}!}\big(\prod_{i'\in[k_i],j'\in\{0,\dots,o(J_i)\}}\frac{((\r\wedge I^f)_{ij}(\q_i\wedge J_i^f)_{i'j'})!}{((\q_i\wedge J_i^f)_{i'j'}!)^{(\r\wedge I^f)_{ij}}}\big)\bigg)\\
  \phi_{\mu\big(I^f\otimes\big(\bigotimes_{i\in[p],j\in\{0,\dots,o(I)\}}(J_i^f)^{\otimes (\r\wedge I^f)_{ij}}\big)\big),(\r\wedge I^f)\diamond(\q_i\wedge J_i^f)_{i\in[p],j\in\{0,\dots,o(I)\}}}\\
  (a_{ii'})_{i\in[p],j\in\{0,\dots,o(I)\},i'\in[k_i],j'\in\{0,\dots,o(J_i)\}}.
\end{multline*}
\Cref{remLcomp} indicates that $\mu\bigg(I^f\otimes\Big(\bigotimes_{i\in[p],j\in\{0,\dots,o(I)\}}(J_i^f)^{\otimes (\r\wedge I^f)_{ij}}\Big)\bigg)$ is constant on\newline$(\r\wedge I^f)\diamond(\q_i\wedge J_i^f)_{i\in[p],j\in\{0,\dots,o(I)\}}$, and so, that $(\r\wedge I^f)\diamond(\q_i\wedge J_i^f)_{i\in[p],j\in\{0,\dots,o(I)\}}$ is a refinement of $\bigg((\r\diamond(\q_i)_{i\in[p]})\wedge \mu\Big(I\otimes\big(\bigotimes_iJ_i^{\otimes r_i}\big)\Big)\bigg)_{ii'k}$.
On one hand,
$$\val{\big((\r\diamond(\q_i)_{i\in[p]})\wedge \mu(I\otimes\bigotimes_iJ_i^{\otimes r_i})\big)_{ii'k}}=\sum_{j+j'=k}(\r\wedge I^f)_{ij}(\q_i\wedge J_i^f)_{i'j'},$$
and on the other hand, $\val{\big((\r\wedge I^f)\diamond(\q_i\wedge J_i^f)_{i\in[p]}\big)_{iji'j'}}=(\r\wedge I^f)_{ij}(\q_i\wedge J_i^f)_{i'j'}$. So, according to Relation \ref{relLrepet}
\begin{multline*}
  \theta_{[I]_{\r},\r}(\theta_{[J_i]_{\q_i},\q_i}(a_{ij})_{j\in[k_i]})_{i\in[p]}=\\
  \Bigg(\frac{\prod_{k\in\N}\Big(\sum_{j+j'=k}(\r\wedge I^f)_{ij}(\q_i\wedge J_i^f)_{i'j'}\Big)!}{\prod_{i\in[p],j\in\{0,\dots,o(I)\}}(\r\wedge I^f)_{ij}!\prod_{i'\in[k_i],j'\in\{0,\dots,o(J_i)\}}((\q_i\wedge J_i^f)_{i'j'}!)^{(\r\wedge I^f)_{ij}}}\Bigg)\\
  \phi_{\omega^{(\r\wedge I^f)\hat \diamond(\q_i\wedge J_i^f)_{i\in[p]}}_{\r\diamond(\q_i)_i\wedge\mu\left(I\otimes\left(\bigotimes_iJ_i^{r_i}\right)\right)}\mu\left(I^f\otimes\left(\bigotimes_{i\in[p]}(J_i^f)^{\otimes \r_i}\right)\right),\r\diamond(\q_i)_i\wedge\mu\left(I\otimes\left(\bigotimes_iJ_i^{r_i}\right)\right)}
  (a_{ii'})_{i\in[p],i'\in[k_i]},
\end{multline*}
\begin{multline*}
  =\Bigg(\frac{\prod_{k\in\N}\Big(\sum_{j+j'=k}(\r\wedge I^f)_{ij}(\q_i\wedge J_i^f)_{i'j'}\Big)!}{\prod_{i\in[p],j\in\{0,\dots,o(I)\}}(\r\wedge I^f)_{ij}!\prod_{i'\in[k_i],j'\in\{0,\dots,o(J_i)\}}((\q_i\wedge J_i^f)_{i'j'}!)^{(\r\wedge I^f)_{ij}}}\Bigg)\\
  \theta_{\left[\mu\left(I\otimes\left(\bigotimes_iJ_i^{\otimes r_i}\right)\right)\right]_{\r\diamond(\q_i)_i\wedge\mu\left(I\otimes\left(\bigotimes_iJ_i^{r_i}\right)\right)},\r\diamond(\q_i)_i\wedge\mu\left(I\otimes\left(\bigotimes_iJ_i^{r_i}\right)\right)}(a_{ii'})_{i\in[p],i'\in[k_i]}.
\end{multline*}
Since
\begin{align*}
  \val{Stab_{\r\diamond(\q_i)_i}\bigg(\mu\Big(I\otimes\big(\bigotimes_iJ_i^{\otimes r_i}\big)\Big)\bigg)}&=\val{(\r\diamond_i\q_i)\wedge \mu\bigg(I\otimes\Big(\bigotimes_iJ_i^{\otimes r_i}\Big)\bigg)}\\
  &=\prod_{k\in\N}\Big(\sum_{j+j'=k}(\r\wedge I^f)_{ij}(\q_i\wedge J_i^f)_{i'j'}\Big)!,
\end{align*}
Relation \ref{rel'comp} is satisfied.

Moreover, if $A'$ is a step algebra and if $f:A\to A'$ is a morphism of step algebras, then, following the definition of $\theta_{[I]_\r,\r}$, the map $f$ is compatible with the operations $\theta_{[I]_\r,\r}$, and so, $f$ is a morphism in $\gamma'(Lev)$. Hence,
\begin{eqnarray*}
  V:&Step&\to \gamma'(Lev)\\
  &(A,\phi)&\mapsto(A,\theta)
\end{eqnarray*}
is a functor.
\subsubsection*{Step 3: \texorpdfstring{$V\circ U=id_{\gamma'(Lev)}$}{GF=id}}

  Let $(A,\gamma)$ be an object in $\gamma'(Lev)$. One sets:
  $$\varphi_{h,\r}:=\gamma_{[h^P]_\r,\r},$$
  and:
  $$\theta_{[I]_\r,\r}(a_1,\dots,a_p):=\varphi_{I^f,\r\wedge I^f}(\underset{o(I)+1}{\underbrace{a_1,\dots,a_1}},\dots,\underset{o(I)+1}{\underbrace{a_p,\dots,a_p}}),$$
  where $I^f$ is a representative of $[I^f]_\r$ which is non-decreasing on each $\r_i$. Then one has, for all $\r\in \Comp_p(n)$ and $[I]_\r\in\∑_{\r}\backslash\L(n)$, $\theta_{[I]_\r,\r}=\gamma_{[I]_\r,\r}$. Indeed, by definition, and following Relation \ref{rel'repet},
  \begin{align*}
    \theta_{[I]_\r,\r}(a_1,\dots,a_p)&=\gamma_{[I]_{\r\wedge I^f},\r\wedge I^f}(\underset{o(I)+1}{\underbrace{a_1,\dots,a_1}},\dots,\underset{o(I)+1}{\underbrace{a_p,\dots,a_p}})
    =\frac{\val{Stab_{\r}(I)}}{\val{Stab_{\r\wedge I^f}(I)}}\gamma_{[I]_\r,\r}(a_1,\dots,a_p),
  \end{align*}
   and $Stab_\r(I)=Stab_{\r\wedge I^f}(I)=\∑_\r\cap\∑_I$.

\subsubsection*{Step 4: \texorpdfstring{$U\circ V=id_{Step}$}{FG=id}}

  Let $(A,\phi)$ be a step algebra. One sets:
  $$\theta_{[I]_\r,\r}(a_1,\dots,a_p):=\phi_{I^f,\r\wedge I^f}(\underset{o(I)+1}{\underbrace{a_1,\dots,a_1}},\dots,\underset{o(I)+1}{\underbrace{a_p,\dots,a_p}}),$$
  and:
  $$\varphi_{h,\r}:=\theta_{[h^P]_\r,\r}.$$
  Then, one has, for all $\r\in \Comp_p(n)$ and $h\in\C_\r$, $\varphi_{h,\r}=\phi_{h,\r}$.

  Indeed, by definition, one has:
  \begin{equation*}
    \varphi_{h,\r}(a_1,\dots,a_p)=\phi_{h,\r\wedge h}(\underset{o(h)+1}{\underbrace{a_1,\dots,a_1}},\dots,\underset{o(h)+1}{\underbrace{a_p,\dots,a_p}}).
  \end{equation*}
  Since $h\in\C_\r$, one then has
  $$(\r\wedge h)_{ij}=\left\{\begin{array}{ccccc}
    r_i, &\mbox{if}&j=h(k)&\mbox{où}&k\in\r_i,\\
    0,&\mbox{else,}
  \end{array}\right.$$
  and so, according to relation \ref{relL0},
  $$\phi_{h,\r}(a_1,\dots,a_p)=\varphi_{h,\r}(a_1,\dots,a_p).$$
\begin{lemm}\label{lemmOmega}
Let $\r\in \Comp_p(n)$, let $(l,m)\in\Comp_2(r_1)$ and let $h:[n]\to[p]$ be a map non decreasing on each $\r_i$. Then the set $Sh(l,m)h:=\{\sigma\cdot h:\sigma\in Sh(l,m)\}$ is a set of representatives of $\∑_{\r\circ_1(l,m)}\backslash\Omega_{\∑_\r}(h)$. Moreover, an element $g$ of $Sh(l,m)h$ is uniquely determined by the families of integers $(\nu_{1j})_{0\le j\le o(h)}=\val{[l]\cap g^{-1}(\{j\})}$ and $(\nu_{2j})_{0\le j\le o(h)}=\val{l+[m]\cap g^{-1}(\{j\})}$, and those families satisfy $\nu_{1j}+\nu_{2j}=\val{\r_1\cap h^{-1}(\{j\})}$ for all $j\in\{0,\dots,o(h)\}$.
\end{lemm}
\begin{proof}
  The first assertion comes from the fact that $Sh(l,m)$ is a system of representative of $\∑_{\r\circ_1(l,m)}\backslash\∑_\r$. The second comes from the fact that all $\sigma\cdot h$ with $\sigma\in Sh(l,m)$ is non-decreasing on $[l]$ and $l+[m]$, and satisfies:
  $$\val{\r_1\cap (\sigma\cdot h)^{-1}(\{j\})}=\val{\r_1\cap h^{-1}(\{j\})}.$$
\end{proof}

\section{Divided power \texorpdfstring{$Lev$}{Lev}-algebras: examples and relations with other types of algebras}\label{secStep}
The aim of this section is to give the first examples of $\Gamma(Lev)$-algebras, and to shed light on natural connections between $\Gamma(Lev)$-algebras and other types of algebras.

In \Cref{othertypes}, we will focus on the links between $\Gamma(Lev)$-algebras and other types of algebras, those links being functorially deduced from operad morphisms to or from the operad $Lev$. 

In \Cref{SBHsec}, more details are given on the free level algebra, which corresponds to the free ``depth-invariant'' algebra of Davis (see \cite{DD}). We show that this level algebra can also be identified with the vector space spanned by the set of binary Huffman sequences (following the definition of \cite{EHP}), on which we define a level multiplication.

In this section, following the result of \Cref{theoL}, we will completely identify $\Gamma(Lev)$-algebras with elements of $Step$. We will use the definitions and notation from \Cref{modop}.

\subsection{Relations with other types of algebras}\label{othertypes}
In this Subsection, we consider two ``forgetful functors'', one from $\Gamma(Lev)$-algebras to $Lev$-algebras, and the other from $Lev$-algebras to $Com$-algebras. We will use the results from \Cref{secLEV} and \Cref{secCOM}.

For a reduced operad $\P$, there is a functor $Tr^*:\Gamma(\P)_{alg}\to \P_{alg}$, deduced from the norm map $Tr:S(\P,\cdot)\to\Gamma(\P,\cdot)$ (see \Cref{Tr} and \cite{BF}, 1.1.18). For $\P=Lev$ it expresses as:
\begin{prop}\label{propsteplev}
  A $\Gamma(Lev)$-algebra $(A,\phi)$ is endowed with a structure of level algebra $(A,*)$ such that:
  $$a*b=\phi_{h,(1,1)}(a,b),$$
  where $h:[2]\to\N$ is the map sending both $1$ and $2$ to $1$.
\end{prop}
\begin{proof}
 Let $(A,\phi)$ be a $\Gamma(Lev)$-algebra. For all $a,b\in A$, denote by $a*b=\phi_{h,(1,1)}(a,b)$. The map $h$ and the composition $(1,1)$ being stable under the action of $\∑_2$, $*$ is commutative. Moreover, for $a,b,c\in A$ and $\lambda\in\F$, one then has
  \begin{align*}
    (a+\lambda b)*c&\overeg{Relation \ref{relLsomme}}\phi_{h,(0,1,1)}(a,\lambda b,c)+\phi_{h,(1,0,1)}(a,\lambda b,c)\\
    &\overeg{Relation \ref{relL0}}\phi_{h,(1,1)}(\lambda b,c)+\phi_{h,(1,1)}(a,c)\\
    &\overeg{Relation \ref{relLlambda}}\lambda(b*c)+a*c.
  \end{align*}
  Hence $*$ is bilinear. Finally, let $g:[4]\to\N$ be the constant map equal to 2. One has $g=\mu(h\otimes h\otimes h)$. From Relation \ref{relLcomp}, one checks that
  $$(a*b)*(c*d)=\phi_{g,(1,1,1,1)}(a,b,c,d).$$
  The map $g$ and the composition $(1,1,1,1)$ are stable under the action of $\∑_4$, so
  \begin{equation*}
    (a*b)*(c*d)
    =(a*c)*(b*d),
  \end{equation*}
  and so, $*$ is indeed a level algebra multiplication.
\end{proof}
For $\P,\P'$ two operads, any morphism of operads $\varphi:\P\to\P'$ induces a functor $\varphi^*:\P'_{alg}\to \P_{alg}$ (see \cite{LV}, 5.2.4) and, through the norm map, a functor $f^*:\Gamma(\P')_{alg}\to \Gamma(\P)_{alg}$. One then has the following result.
\begin{prop}
  A divided power algebra $(A,\gamma)$ is endowed with a structure of $\Gamma(Lev)$-algebra $(A,\phi)$ such that
  \begin{eqnarray*}
    \phi_{h,\r}:&A^{\times p}&\to A\\
    &(a_1,\dots,a_p)&\mapsto \prod_{i=1}^p\gamma_{r_i}(a_i),
  \end{eqnarray*}
  where the product is given by the commutative algebra structure on $A$, and where, by convention, the terms such that $r_i=0$ are omitted.
\end{prop}
\begin{proof}
  There is an operad morphism $pr:Lev\to Com$ sending $I\in \L(n)$ to $X_n$. This morphism induces a functor $pr^*:\Gamma(Com)_{alg}\to \Gamma(Lev)_{alg}$. If $(A,f)$ is a $\Gamma(Com)$-algebra, then $pr^*(A,f)=(A,f\circ \Gamma(pr,A))$. Let $(A,\gamma)$ be a divided power algebra (as in \Cref{puissdiv}). For $\r\in\Comp_p(n)$, $(a_1,\dots,a_p)\in A^{\times p}$, define
  $$\beta_{X_n,\r}(a_1,\dots,a_p)=\prod_{i\in[p]}\gamma_{r_i}(a_i),$$
  as in \ref{eqCOM}. This gives an object $(A,\beta)$ of $\beta'(Com)$. According to \Cref{deff}, there is a morphism $f:\Gamma(Com,A)\to A$, such that $(A,f)$ is a $\Gamma(Com)$-algebra, that satisfies, for all $\r\in\Comp_p(n)$, $(a_1,\dots,a_p)\in A^{\times p}$:
  $$f\bigg(\sum_{\sigma\in\∑_n/\∑_\r}X_n\otimes\sigma\cdot\Big(\bigotimes_{i=1}^pa_i^{r_i}\Big)\bigg)=\beta_{X_n,\r}(a_1,\dots,a_p).$$
  One then obtains a $\Gamma(Lev)$-algebra $(A,f\circ\Gamma(pr,A))$. Note that, for all $\r\in\Comp_p(n)$, $I\in \L(n)^{\∑_\r}$ and $(a_1,\dots,a_p)\in A^{\times p}$,
  $$f\bigg(\Gamma(pr,A)\Big(\sum_{\sigma\in\∑_n/\∑_\r}\sigma\cdot I\otimes \sigma\cdot\big(\bigotimes_{i=1}^pa_i^{\otimes r_i}\big)\Big)\bigg)=f\bigg(\sum_{\sigma\in\∑_n/\∑_\r}X_n\otimes\sigma\cdot\Big(\bigotimes_{i=1}^pa_i^{r_i}\Big)\bigg)$$
  This also gives an object of $\gamma'(Lev)$ such that, for all $I\in \L(n)$, $\r\in\Comp_p(n)$, $(a_1,\dots,a_p)\in A^{\times p}$,
  $$\gamma_{[I]_\r,\r}(a_1,\dots,a_p)=f\bigg(\Gamma(pr,A)\Big(\sum_{\sigma\in\∑_n/\∑_\r}\sigma\cdot \O_\r(I)\otimes \sigma\cdot\big(\bigotimes_{i=1}^pa_i^{\otimes r_i}\big)\Big)\bigg)$$
  According to the previous Section, for all $\r\in\Comp_p(n)$, $h\in\C_\r$ and $(a_1,\dots,a_p)\in A^{\times p}$, since $h^P\in\L(n)^{\∑_\r}$, one has (see \ref{eqphi})
  \begin{align*}
    \phi_{h,\r}(a_1,\dots,a_p)&=\gamma_{[h^P]_\r,\r}(a_1,\dots,a_p)=f\bigg(\Gamma(pr,A)\Big(\sum_{\sigma\in\∑_n/\∑_\r}\sigma\cdot h^P\otimes \sigma\cdot\big(\bigotimes_{i=1}^pa_i^{\otimes r_i}\big)\Big)\bigg)\\
    &=f\bigg(\sum_{\sigma\in\∑_n/\∑_\r}X_n\otimes\sigma\cdot\Big(\bigotimes_{i=1}^pa_i^{r_i}\Big)\bigg)=\beta_{X_n,\r}(a_1,\dots,a_p)=\prod_{i\in[p]}\gamma_{r_i}(a_i)\qedhere
  \end{align*}
\end{proof}

\subsection{The free divided power level algebra on one generator}\label{SBHsec}
This section aims to give an explicit description of the free $\Gamma(Lev)$-algebra generated by one element. 

\Cref{remlib1} explains why the free $Lev$-algebra and the free $\Gamma(Lev)$-algebra have the same underlying vector space. In \cite{DD}, D. M. Davis shows that the underlying vector space of the free level algebra generated by one element is endowed with a structure of module over the Steenrod algebra. He denote this module by $Y_1$, and this module coincides with the unstable module $X_1$ studied by Carlsson in \cite{GC}. The non-associative multiplication used by Carlsson on $X_1$ coincides with its level algebra product. 

In \Cref{propfreeLEV}, we give a new characterisation of this object, using binary Huffman sequences of \cite{EHP}. \Cref{propstepBHS} describes the free divided level operations on this object.

 \begin{rema}\label{remlib1}
  Let $\M$ be a $\∑$-module such that for all $n\in\N$, $\M(n)$ is equipped with a basis stable under the action of $\∑_n$. According to \Cref{B}, there is, for all $n\in\N$, an isomorphism $\O_{\∑_n}:\M(n)_{\∑_n}\to\M(n)^{\∑_n}$. Denote by $\F x$ the one dimensional $\F$-vector space spanned by $x$. One has
  \begin{equation*}
    \hspace{-8.36726pt}\Gamma(\M,\F x)=\bigoplus_{n\in\N}(\M(n)\otimes (\F x)^{\otimes n})^{\∑_n}
    \cong\bigoplus_{n\in\N}\M(n)^{\∑_n}
    \overset{\bigoplus_{n\in\N}\O_{\∑_n}^{-1}}{\cong}\bigoplus_{n\in\N}\M(n)_{\∑_n}
    \cong\bigoplus_{n\in\N}(\M(n)\otimes (\F x)^{\otimes n})_{\∑_n}.
  \end{equation*}
  In particular, when $\P$ is a reduced operad equipped with a basis stable under the action of the symmetric groups, then the free $\Gamma(\P)$-algebra over one generator is isomorphic, as a vector space, to the free $\P$-algebra over one generator. From \Cref{propfreeLEV}, we readily deduce:
\end{rema}
\begin{defi}\label{defhuf}
  For all $n\in\N$, let $\BHS(n)$ denote the set 
  $$\{u\in \N^\N:\sum_{i}u(i)=n\mbox{ and }\sum_{i\in\N}\frac{u(i)}{2^{i}}=1\}.$$
  A sequence $u\in \BHS(n)$ is called a binary Huffman sequence (following \cite{EHP}, Definition 2). Denote by $\BHS$ the union $\bigsqcup_{n\ge 1}\BHS(n)$. For all $u,v\in \BHS$, set $u\cdot v:=(0,u(0)+v(0),u(1)+v(1),\dots)\in\BHS$.
\end{defi}
\begin{prop}\label{prophuf}
  For all $n>0$, $\∑_n\backslash \L(n)$ is in bijection with $\BHS(n)$.
\end{prop}
\begin{proof}
  For all $n\in\N$, let $\pi:\L(n)\to \BHS(n)$ be the map sending $I$ to $(\val{I_0},\val{I_1},\dots)$. The map $\pi$ is surjective because, for $u\in\BHS(n)$, the composition $\u$ is in $\L(n)$, and $\pi(\u)=u$. Moreover, $I,J\in Lev(n)$ satisfy $\pi(I)=\pi(J)$ if and only if there exists $\sigma\in\∑_n$ such that $J=\sigma\cdot I$. So, $\pi$ induces a bijection $\bar\pi:\∑_n\backslash\L(n)\to \BHS(n)$.
\end{proof}
\begin{rema}\label{remhuf}
  Recall that the set operad $\L$ encodes level multiplications on sets. It is easy to see that the level multiplication on $\BHS$ endowed by the identification of \Cref{prophuf} is the multiplication $(u,v)\mapsto u\cdot v$ of \Cref{defhuf}.
\end{rema}
\begin{prop}\label{propfreeLEV}
  The free level algebra on one generator is isomorphic, as an $\F$-vector space, to the vector space spanned by the set of binary Huffman sequences. The level multiplication on $\F[\BHS]$ is spanned by the level multiplication $(u,v)\mapsto u\cdot v$ defined in \Cref{defhuf}. 
\end{prop}
\begin{proof}
  The identification of \Cref{prophuf} linearly extend to an isomorphism $\varphi$ and we have:
  \begin{equation*}
    S(Lev,\F x)=\bigoplus_{n\in\N}Lev(n)_{\∑_n}
    =\bigoplus_{n\in\N}\F[\∑_n\backslash \L(n)]
    \overset{\varphi}{\cong}\bigoplus_{n\in\N}\F[\BHS(n)]
    =\F[\BHS].\qedhere
  \end{equation*}
  Following \Cref{remhuf}, this gives a level algebra isomorphism between $S(Lev,\F x)$ and $\F[\BHS]$ endowed with the multiplication $(u,v)\mapsto u \cdot v$.
\end{proof}
\begin{coro}\label{propfreestep}
  The free $\Gamma(Lev)$-algebra on one generator is isomorphic, as an $\F$-vector space, to the vector space spanned by the set of binary Huffman sequences.
\end{coro}
\begin{rema}
  \Cref{propfreeLEV} and \Cref{propfreestep} give two different structures on $\F[\BHS]$: a structure of level algebra, and one of divided power level algebra. The former is given by the level multiplication $(u,v)\mapsto u\cdot v$. The latter is given by the isomorphism $\F[\BHS]\to \Gamma(Lev,\F x)$, that associates $u\in\BHS(n)$ with $\O_{\∑_n}(\u)\otimes x^{\otimes n}$. By \Cref{propsteplev}, this $\Gamma(Lev)$-algebra structure yields another level multiplication on $\F[\BHS]$. We will denote it by $(u,v)\mapsto u*v$.
\end{rema}
\begin{lemm}\label{lemlev}
  The level algebra product given by \Cref{propsteplev} on $\F[\BHS]$ seen as the free $\Gamma(Lev)$-algebra spanned by one generator is given by
  $$u*v=\prod_{j\in\N}\binom{u(j)+v(j)}{u(j)}(u\cdot v)$$
\end{lemm}
\begin{proof}
  Let $u\in \BHS(n)$, $v\in \BHS(m)$. \Cref{propsteplev} implies that the level algebra product on $\F[\BHS]$ as a $\Gamma(Lev)$-algebra is given by:
  $$u*v=\phi_{h,(1,1)}(u,v)$$
  with $h:[2]\to\N$ constant and equal to $1$.

  Unwinding the definitions, one gets:
   \begin{align*}
    \phi_{h,(1,1)}(u,v)&=\phi_{h,(1,1)}(u,v)\\
    &=\gamma_{[h^P]_{(1,1)}},(1,1)(u,v)\\
    &\overeg{\ref{eqgamma}}\beta_{\O_{\∑_{(1,1)}}(h^P),(1,1)}(u,v)\\
    &=\tilde\mu\Big(\sum_{\sigma\in\∑_2/\∑_{(1,1)}}\sigma\cdot h^P\otimes\sigma\cdot\big((\O_{\∑_n}(\u)\otimes x^{\otimes n})^{\otimes 1}\otimes(\O_{\∑_m}(\v)\otimes x^{\otimes m})^{\otimes 1}\big)\Big)
  \end{align*}
  According to \Cref{A}, one then gets:
  \begin{align*}
    \phi_{h,(1,1)}(u,v)&=\sum_{\sigma\in\∑_{m+n}/\∑_{1}\wr\∑_{\u}\times\∑_{1}\wr\∑_{\v}}\sigma\cdot\mu(h^P\otimes u^{\otimes 1}\otimes v^{\otimes 1})\otimes \sigma\cdot x^{\otimes m+n}\\
    &=\frac{\val{Stab\big(\mu(h^P\otimes \u\otimes\v)\big)}}{\val{\∑_{\u}}\val{\∑_{\v}}}\O_{\mu(h^P\otimes \u\otimes \v)}\big(\mu(h^P\otimes \u\otimes \v)\big)\otimes x^{m+n}\\
    &=\prod_{j\in\N}\binom{u(j)+v(j)}{u(j)}(u\cdot v)\qedhere
  \end{align*}
\end{proof}
The following lemma shows that $\F[\BHS]$ is indeed generated, as a $\Gamma(Lev)$-algebra, by one element:
\begin{lemm}
  For all $u\in \BHS(n)$, one has:
  $$u=\phi_{\u^f,u}(\underset{o(u)+1}{\underbrace{1_\L,\dots,1_\L}}),$$
  where $1_\L=(1,0,0,\dots)\in\BHS(1)$ is the element associated with $1_\L\otimes x^{\otimes 1}$.
\end{lemm}
\begin{proof}
  Following the computations of the proof of \Cref{lemlev}, one has:
  \begin{align*}
    \phi_{\u^f,u}(\underset{o(u)+1}{\underbrace{1_\L,\dots,1_\L}})&=\tilde\mu\bigg(\sum_{\sigma\in\∑_n/\∑_{\u}}\sigma\cdot\O_{\∑_{\u}}(\u)\otimes \sigma\cdot\Big(\bigotimes_{i=0}^{o(u)}(1_{\L}\otimes x)^{\otimes u(i)}\Big)\bigg)\\
    &=\sum_{\sigma\in\∑_n/\prod_{i=1}^p\∑_{u(i)}\wr\∑_1}\sigma\cdot\mu(\u\otimes 1_\L^{\otimes n})\otimes x^{\otimes n}\\
    &=\O_{\∑_n}(\u)\otimes x^{\otimes n}=u\qedhere
  \end{align*}
\end{proof}
Finally, we get:
\begin{prop}\label{propstepBHS}
  The free divided power level algebra is the vector space $\F[\BHS]$ equipped with the polynomial operations:
$$\phi_{h,\r}(u_1,\dots,u_p)=\frac{\prod_{l\in\N}\Big(\sum_{j\in[p]}r_ju_j(l-k_j)\Big)!}{\prod_{i=1}^pr_i!\prod_{k\in\N}(u_i(k))!^{r_i}}u,$$
with $\{k_i\}=h(\r_i)$, $u=\Big(\sum_{j\in[p]}r_ju_j(l-k_j)\Big)_{l\in\N}$, and where, for $x<0$, $u_i(x)=0$.
\end{prop}
\begin{proof}

 Let $\r\in \Comp_p(n)$, $h\in\C_\r$ and for all $i\in[p]$, let $u_i\in \BHS(m_i)$. Following the computations of the proof of \Cref{lemlev}, one has:
  \begin{align*}
    \phi_{h,\r}(u_1,\dots,u_p)=\tilde\mu\bigg(\sum_{\sigma\in\∑_n/\∑_\r}\sigma\cdot\O_\r(h^P)\otimes\sigma\cdot\Big(\bigotimes_{i=1}^p(\O_{\∑_{m_i}}(\u_i)\otimes x^{\otimes m_i})^{\otimes r_i}\Big)\bigg)
  \end{align*}
  Denote by $M=\sum_{i=1}^pr_im_i$. According to \Cref{A}, one then gets:
  \begin{multline*}
    \phi_{h,\r}(u_1,\dots,u_p)=\sum_{\sigma\in\∑_M/\prod_{i=1}^p\∑_{r_i}\wr\∑_{\u_i}}\sigma\cdot\mu(h^P\otimes \u_1^{\otimes r_1}\otimes\dots\otimes \u_p^{\otimes r_p})\otimes \sigma\cdot x^{\otimes M}\\
    =\sum_{\sigma\in Stab(\mu(h^P\otimes \u_1^{\otimes r_1}\otimes\dots\otimes \u_p^{\otimes r_p}))/\prod_{i=1}^p\∑_{r_i}\wr\∑_{\u_i}}\sigma\cdot\O_{\∑_M}(\mu(h^P\otimes \u_1^{\otimes r_1}\otimes\dots\otimes \u_p^{\otimes r_p})\otimes \sigma\cdot x^{\otimes M})\\
    =\frac{\val{Stab(\mu(h^P\otimes \u_1^{\otimes r_1}\otimes\dots\otimes \u_p^{\otimes r_p}))}}{\prod_{i=1}^pr_i!\prod_{k\in\N}(u_i(k)))!^{r_i}}\O_{\∑_M}(\mu(h^P\otimes \u_1)^{\otimes r_1}\otimes\dots\otimes \u_p)^{\otimes r_p})\otimes x^{\otimes M}
  \end{multline*}
  And, since $\mu\bigg(h^P\otimes\Big(\bigotimes_{i\in[p]}\u_i^{\otimes r_i}\Big)\bigg)=\Big(\sum_{j\in[p]}r_ju_j(l-k_j)\Big)_{l\in\N}$, one gets the expected result.
  \end{proof}

  Let us give an example. Let $h:[2]\to\N$ be constant and equal to $1$. Then, as explained in \Cref{propsteplev} and \Cref{lemlev}, $\phi_{h,(1,1)}$ plays the role of the level algebra multiplication in $\F[\BHS]$ seen as the free $\Gamma(Lev)$-algebra on one generator, and for all $u\in \BHS(n)$, one has
  $$u*u=\phi_{h,(1,1)}(u,u)=\prod_{j\in\N}\binom{2u(j)}{u(j)}(u\cdot u)$$
  Note that the coefficient $\prod_{j\in\N}\binom{2u(j)}{u(j)}$ is even. Using \Cref{propstepBHS}, one also has:
  
    $$\phi_{h,(2)}(u)=\frac{\prod_{l\in\N}(2u(l-1))}{2!\prod_{k\in\N}(u(k))!^2}(u\cdot u)=\frac{1}{2}\prod_{j\in\N}\binom{2u(j)}{u(j)}(u\cdot u)$$

  In this way, $\phi_{h,(2)}$ plays the role of a divided square in any $\Gamma(Lev)$-algebra.


\appendix
\section{Proof of \texorpdfstring{\Cref{propA}}{}}\label{A}
This section aims to clarify the product of the monad $\Gamma(\P)$, when $\P$ is a reduced operad. This is the result of \Cref{propA}.

Let us write the proposition again. Let $\P$ be a reduced operad and $V$ be an $\F$-vector space. Denote by $\mu:\P\circ\P\to\P$ the multiplication of $\P$ and by $\tilde\mu:\Gamma(\P)\circ\Gamma(\P)\to\Gamma(\P)$ the multiplication of the monad $\Gamma(\P)$. Consider an element $\t\in\Gamma(\P,\Gamma(\P,V))$ of the form
\begin{equation*}
  \t=\sum_{\sigma\in \∑_n/\∑_\r}\sigma\cdot x\otimes \sigma\cdot\bigg(\bigotimes_{i=1}^p\Big(\sum_{\sigma_i\in\∑_{m_i}/\∑_{\q_i}}\sigma_i\cdot x_i\otimes \sigma_i\cdot\_{v_i}\Big)^{\otimes r_i}\bigg),
\end{equation*}
with $\r\in\Comp_p(n)$, $x\in \P(n)^{\∑_\r}$, and for all $i\in[p]$, $\q_i\in\Comp_{k_i}(m_i)$, $x_i\in\P(m_i)^{\∑_{\q_i}}$ and $\_{v_i}\in (V^{\otimes m_i})^{\∑_{\q_i}}$.
\begin{prop}
  One has
\begin{equation*}
 \tilde\mu_V(\t)=\sum_{\tau\in\∑_{M}/\prod_{i=1}^p\∑_{r_i}\wr\∑_{\q_i}}\mu(x\otimes[\tau\otimes x_1^{\otimes r_1}\otimes\dots\otimes x_p^{\otimes r_p}]_{\prod_{i}\∑_{m_i}^{\times r_i}})\otimes \tau\cdot(\_{v_1}^{\otimes r_1}\otimes\dots\otimes\_{v_p}^{\otimes r_p}),
\end{equation*}
where $M=r_1m_1+\dots+r_pm_p$, and 
$$[\tau\otimes x_1^{\otimes r_1}\otimes\dots\otimes x_p^{\otimes r_p}]_{\prod_{i}\∑_{m_i}^{\times r_i}}\in \Ind_{\prod_i\∑_{m_i}^{\times r_i}}^{\∑_M}\bigotimes_{i}\P(m_i)^{\otimes r_i}.$$
\end{prop}
\begin{proof}
  The multiplication $\tilde{\mu}_V:\Gamma(\P,\Gamma(\P,V))\to\Gamma(\P,V)$ is the composition:
\begin{equation*}
  \hspace{-20pt}\diag{
  \bigoplus_{n\ge 0,M\ge 0}(\P(n)\otimes\P^{\otimes n}(M)\otimes V^{\otimes M})^{\∑_n,\∑_M}\ar[r]&\bigoplus_{M\ge0}((\bigoplus_{n\ge 0}\P(n)\otimes\P^{\otimes n}(M))^{\∑_n}\otimes V^{\otimes M})^{\∑_M}\ar[d]^{Tr^{-1}}\\
  \bigoplus_{n\ge 0}(\P(n)\otimes\Gamma(\P^{\otimes n},V))^{\∑_n}\ar[u]&\bigoplus_{M\ge0}((\bigoplus_{n\ge 0}\P(n)\otimes\P^{\otimes n}(M))_{\∑_n}\otimes V^{\otimes M})^{\∑_M}\ar[d]^{\mu}\\
   \bigoplus_{n\ge 0}(\P(n)\otimes \Gamma(\P,V)^{\otimes n})^{\∑_n}\ar[u]_{\varphi}&(\bigoplus_{M\ge 0}\P(M)\otimes V^{\otimes M})^{\∑_M}
   },
\end{equation*}
 The isomorphism $\varphi:(\P(n)\otimes \Gamma(\P,V)^{\otimes n})^{\∑_n}\to (\P(n)\otimes \Gamma(\P^{\otimes n},V))^{\∑_n}$ (see \cite{NS}, Proposition 4.2) is given by
\begin{multline*}
  \varphi(\t)=\sum_{\sigma\in \∑_n/\∑_\r}\sigma\cdot x\otimes \sigma\cdot\sum_{\tau\in\∑_M/\prod_{i=1}^p\∑_{m_i}^{\times r_i}}\\
  \sum_{(\sigma_{i,j})_{i\in[p],j\in[r_i]}\in \prod_{i=1}^p(\∑_{m_i}/\∑_{\q_i})^{\times r_i}}[\tau\otimes\Big(\bigotimes_{i=1}^p\bigotimes_{j=1}^{r_i}\sigma_{i,j}\cdot x_i\Big)]_{\prod_{i}\∑_{m_i}^{\times r_i}}\otimes \tau\cdot\Big(\bigotimes_{i=1}^p\bigotimes_{j=1}^{r_i}\sigma_{i,j}\cdot \_{v_i}\Big),
\end{multline*}
which is equal to
\begin{equation*}
  \sum_{\sigma\in \∑_n/\∑_{\r}}\sigma\cdot x\otimes \sigma\cdot\bigg(\sum_{\tau\in\∑_{M}/\prod_{i=1}^p\∑_{\q_i}^{\times r_i}}[\tau\otimes x_1^{\otimes r_1}\otimes\dots\otimes x_p^{\otimes r_p}]_{\prod_{i}\∑_{m_i}^{\times r_i}}\otimes \tau\cdot(\_{v_1}^{\otimes r_1}\otimes\dots\otimes\_{v_p}^{\otimes r_p})\bigg)
\end{equation*}
\begin{equation*}
  =\sum_{\sigma\in \∑_n/\∑_{\r}}\sum_{\tau\in\∑_{M}/\prod_{i=1}^p\∑_{\q_i}^{\times r_i}}(\sigma\cdot x\otimes \sigma\cdot[\tau\otimes x_1^{\otimes r_1}\otimes\dots\otimes x_p^{\otimes r_p}]_{\prod_{i}\∑_{m_i}^{\times r_i}})\otimes \tau\cdot(\_{v_1}^{\otimes r_1}\otimes\dots\otimes\_{v_p}^{\otimes r_p}).
\end{equation*}
Every $\tau\in\∑_{M}/\prod_{i=1}^p\∑_{\q_i}^{\times r_i}$ can be decomposed as a product $\tau_1\tau_2^*$ with $\tau_1\in\∑_{M}/\prod_{i=1}^p\∑_{r_i}\wr\∑_{\q_i}$ and $\tau_2\in \∑_\r$ acting by block of size $m_i$. Yet, $x$ is invariant under the action of $\tau_2$ and, $x_1^{\otimes r_1}\otimes\dots\otimes x_p^{\otimes r_p}$ and $\_{v_1}^{\otimes r_1}\otimes\dots\otimes\_{v_p}^{\otimes r_p}$ are both invariant under the action of $\tau_2^*$, so, one gets
\begin{multline*}
	\sum_{\sigma\in \∑_n/\∑_\r}\sum_{\tau_2\in \∑_\r}\sum_{\tau_1\in\∑_{M}/\prod_{i=1}^p\∑_{r_i}\wr\∑_{\q_i}}((\sigma\tau_2)\cdot x\otimes\sigma\cdot[\tau_1\tau_2^*\otimes x_1^{\otimes r_1}\otimes\dots\otimes x_p^{\otimes r_p}]_{\prod_{i}\∑_{m_i}^{\times r_i}})\\\otimes \tau_1\tau_2^*\cdot(\_{v_1}^{\otimes r_1}\otimes\dots\otimes\_{v_p}^{\otimes r_p})
\end{multline*}
is equal to
\begin{multline*}
  \sum_{\sigma\in \∑_n/\∑_\r}\sum_{\tau_2\in \∑_\r}\sum_{\tau_1\in\∑_{M}/\prod_{i=1}^p\∑_{r_i}\wr\∑_{\q_i}}((\sigma\tau_2)\cdot x\otimes(\sigma\tau_2)\cdot[\tau_1\otimes x_1^{\otimes r_1}\otimes\dots\otimes x_p^{\otimes r_p}]_{\prod_{i}\∑_{m_i}^{\times r_i}})\\\otimes \tau_1\cdot(\_{v_1}^{\otimes r_1}\otimes\dots\otimes\_{v_p}^{\otimes r_p})
\end{multline*}
$$=\sum_{\sigma\in \∑_n}\sum_{\tau\in\∑_{M}/\prod_{i=1}^p\∑_{r_i}\wr\∑_{\q_i}}(\sigma\cdot x\otimes\sigma\cdot[\tau\otimes x_1^{\otimes r_1}\otimes\dots\otimes x_p^{\otimes r_p}]_{\prod_{i}\∑_{m_i}^{\times r_i}})\otimes \tau\cdot(\_{v_1}^{\otimes r_1}\otimes\dots\otimes\_{v_p}^{\otimes r_p}).$$
Then, applying the inverse of the norm map $Tr_{\P,\P}:\P\circ \P\to \P\ \tilde{\circ}\ \P$, one gets
\begin{equation*}
  \sum_{\tau\in\∑_{M}/\prod_{i=1}^p\∑_{r_i}\wr\∑_{\q_i}}[x\otimes[\tau\otimes x_1^{\otimes r_1}\otimes\dots\otimes x_p^{\otimes r_p}]_{\prod_{i}\∑_{m_i}^{\times r_i}}]_{\∑_n}\otimes \tau\cdot(\_{v_1}^{\otimes r_1}\otimes\dots\otimes\_{v_p}^{\otimes r_p}).
\end{equation*}
Finally, applying the mutliplication of the operad $\P$ gives the expected result.
\end{proof}
\bibliographystyle{plain}

{\scshape Sacha Ikonicoff,} Univ Paris Diderot, Institut de Mathématiques de Jussieu-Paris Rive
Gauche, CNRS, Sorbonne Université, 8 place Aurélie Nemours, F-75013,
Paris, France\newline
E-mail address: sacha.ikonicoff@imj-prg.fr
\end{document}